\documentclass[12pt]{amsart}
\usepackage{amssymb,amsxtra,amsmath,amsfonts}
\usepackage{verbatim}
\usepackage{graphicx,color,graphics} 
\newtheorem{theorem}{Theorem}
\newtheorem{lemma}[theorem]{Lemma}
\newtheorem{proposition}[theorem]{Proposition}
\newtheorem{corollary}[theorem]{Corollary}

\theoremstyle{definition}
\newtheorem{definition}[theorem]{Definition}
\newtheorem{example}[theorem]{Example}

\theoremstyle{remark}
\newtheorem{remark}[theorem]{Remark}

\numberwithin{equation}{section}
\numberwithin{theorem}{section}

\newcommand\thref{Theorem \ref}
\newcommand\leref{Lemma \ref}
\newcommand\prref{Proposition \ref}
\newcommand\coref{Corollary \ref}
\newcommand\deref{Definition \ref}
\newcommand\exref{Example \ref}
\newcommand\reref{Remark \ref}
\renewcommand{\comment}[1]{}
\newcommand{\beq}{\begin{equation}}
\newcommand{\eeq}{\end{equation}}
\newcommand{\beqa}{\begin{eqnarray}}
\newcommand{\eeqa}{\end{eqnarray}}
\newcommand{\beaa}{\begin{eqnarray*}}
\newcommand{\ben}{\begin{eqnarray*}}
\newcommand{\eaa}{\end{eqnarray*}}
\newcommand{\een}{\end{eqnarray*}}
\newcommand{\leftexp}[2]{{\vphantom{#2}}^{#1}{#2}}

\def\CC{\mathcal{C}}
\def\hS{\hat{S}}

\def\cR{\mathcal{R}}
\def\hcR{\hat{\cR}}

\def\W{\mathcal{W}}
\def\A{\mathcal{A}}
\def\B{\mathcal{B}}
\def\C{\mathbb{C}}
\def\CP{\mathbb{CP}}
\def\D{\mathcal{D}}

\def\F{\mathcal{F}}
\def\H{\mathcal{H}}

\def\M{\mathcal{M}}
\def\N{\mathbb{N}}
\def\O{\mathcal{O}}

\def\R{\mathbb{R}}
\def\S{\mathcal{S}}
\def\T{\mathcal{T}}

\def\Z{\mathbb{Z}}
\def\f{\mathbf{f}}

\def\q{\mathbf{q}}
\def\t{\mathbf{t}}

\DeclareMathOperator\Aut{Aut}
\DeclareMathOperator\End{End}

\DeclareMathOperator\Hom{Hom}

\DeclareMathOperator\tr{tr}

\DeclareMathOperator\Ker{Ker}
\DeclareMathOperator\Res{Res}
\DeclareMathOperator\Ind{Ind}

\def\IFF{\Leftrightarrow}
\def\TO{\Rightarrow}
\def\one{{\boldsymbol{1}}}  
\def\vac{{\boldsymbol{1}}}  
\def\pt{\mathrm{pt}}
\def\ii{\sqrt{-1}} 
\def\al{\alpha}
\def\be{\beta}

\def\de{\delta}

\def\io{\iota}
\def\ep{\varepsilon}

\def\la{\lambda}
\def\si{\sigma}
\def\Si{\Sigma}
\def\Om{\Omega}
\def\om{\omega}
\def\ze{\zeta}

\def\tensor{\otimes}
\def\d{\partial}

\def\({\left(}
\def\){\right)}
\def\[{\left[}
\def\]{\right]}
\def\<{\left\langle}
\def\>{\right\rangle}

\def\lieg{{\mathfrak{g}}}
\def\lieh{{\mathfrak{h}}}
\def\liesl{{\mathfrak{sl}}}

\def\gl{\lambda}

\begin{document}

\title[$\W$-Constraints for Simple Singularities]{$\W$-Constraints for the Total Descendant Potential of a Simple Singularity}
\author{Bojko Bakalov}
\address{Department of Mathematics\\
North Carolina State University\\
Raleigh, NC 27695, USA}
\email{bojko\_bakalov@ncsu.edu}

\author{Todor Milanov}
\address{Kavli IPMU\\
University of Tokyo (WPI)\\
Japan}
\email{todor.milanov@ipmu.jp}


\date{March 15, 2012}


\begin{abstract}
Simple, or Kleinian, singularities are classified by Dyn\-kin diagrams of
type $ADE$. Let $\lieg$ be the corresponding finite-dimensional Lie algebra, and $W$ its Weyl group.
The set of $\lieg$-invariants in the basic representation of the affine Kac--Moody algebra
$\hat\lieg$ is known as a $\W$-algebra and is a subalgebra of the 
Heisenberg vertex algebra $\F$. Using period integrals, we construct an analytic
continuation of the twisted representation of $\F$. Our construction yields a global object, which may be called
a $W$-twisted representation of $\F$.
Our main result is that the total descendant potential of the singularity, 
introduced by Givental, is a highest weight vector for the $\W$-algebra.
\end{abstract}

\maketitle


\section{Introduction}\label{s1}

\subsection{Motivation from Gromov--Witten theory}\label{sgw}

Recall that the \emph{Gromov}--\emph{Witten} (GW) \emph{invariants} of a projective manifold $X$ consist of correlators 
\begin{equation}\label{correlator}
\langle \tau_{k_1}(v_{1}),\dots,\tau_{k_n}(v_{n})\rangle_{g,n,d}
\end{equation}
where $v_1,\dots,v_n\in H^*(X;\mathbb{C})$ are cohomology classes and
the enumerative meaning of the correlator is the following. Let
$C_1,\dots,C_n$ be $n$ cycles in $X$ in a sufficiently generic
position that are Poincar\'e dual to $v_1,\dots,v_n$,
respectively. Then the GW invariant \eqref{correlator} counts the
number of genus-$g$, degree-$d$ holomorphic curves in $X$ that are
tangent (in an appropriate sense) to the cycles $C_i$ with multiplicities $k_i$. For the precise definition we refer to \cite{W1, Ko2, Be, LT}.  
After A.\ Givental \cite{G2}, we organize the GW invariants in a generating series $\D_X$ called the {\em total descendant potential} of $X$ and defined as follows. Choose a basis $\{v_i\}_{i=1}^N$ of the vector (super)space $H=H^*(X;\mathbb{C})$ and let $t_k = \sum_{i=1}^N t_k^i v_i \in H$. Then
\begin{equation*}
\D_X (\t)= \exp \Bigl(\sum_{g,n,d} \frac{Q^d}{n!}\ \hbar^{g-1}
\sum_{k_1,\dots,k_n \geq 0}
\langle \tau_{k_1}(t_{k_1}),\dots,\tau_{k_n}(t_{k_n}) \rangle_{g,n,d} \Bigr),
\end{equation*}
where $\t=(t_0,t_1,\dots)=(t_k^i)$ and the definition of the correlator is extended multi-linearly in its arguments.
The function $\D_X$ is interpreted as a formal power series in the variables $t_k^i$ with coefficients formal Laurent series in $\hbar$ whose coefficients are elements of the Novikov ring $\mathbb{C}[Q]$. 

When $X$ is a point and hence $d=0$, the potential $\D_{\pt}$ (also known as the partition function of pure gravity) is a generating function for certain intersection numbers on the Deligne--Mumford moduli space of Riemann surfaces $\overline{\M}_{g,n}$. It was conjectured by Witten \cite{W1} and proved by Kontsevich \cite{Ko1} that $\D_{\pt}$ is a tau-function for the \emph{KdV hierarchy} of soliton equations. (We refer to \cite{Di,vanM} for excellent introductions to soliton equations.) In addition, $\D_{\pt}$ satisfies one more constraint called the \emph{string equation}, 
which together with the KdV hierarchy determines uniquely $\D_{\pt}$ (see \cite{W1}). It was observed in \cite{DVV,FKN,KS} that the tau-function of KdV satisfying the string equation is characterized as the unique solution of $L_n\D_\pt=0$ for $n\geq -1$, where $L_n$ are certain differential operators representing the \emph{Virasoro algebra}. This means that $\D_{\pt}$ is a highest-weight vector for the Virasoro algebra and in addition satisfies the string equation $L_{-1}\D_\pt=0$.

One of the fundamental open questions in Gromov--Witten theory is the Virasoro
conjecture suggested by S.\ Katz and the physicists Eguchi, Hori,
Xiong, and Jinzenji (see \cite{EHX, EJX, DZ3}), which says that $\D_X$
satisfies Virasoro constraints similar to the constraints for $\D_{\pt}$ for a certain representation of the Virasoro algebra. The equation $L_n\D_X=0$ has a simple combinatorial meaning: it gives a rule for simplifying the correlators \eqref{correlator} when $v_1=1$ and $k_1=n+1$.

A natural question is whether the results for $X=\pt$ can be generalized for any projective manifold $X$. In particular, is there an integrable hierarchy that together with the Virasoro constraints will uniquely characterize the GW invariants of $X$? Alternatively, are there other combinatorial rules that will allow us to simplify the correlator \eqref{correlator} for any cohomology class $v_1$, not only for $v_1=1$ (cf.\ \cite{DVV,Goe})? A representation-theoretic interpretation of such rules is that $\D_X$ is a highest-weight vector for an algebra containing the Virasoro algebra. 
Answering these questions in general is a very difficult problem. However, there is a class of manifolds for which the problem can be formulated entirely in the language of differential equations and representation theory.

\subsection{Semi-simple Frobenius manifolds and Givental's formula}\label{sfrobgiv}

Let us denote by $(\cdot,\cdot)$ the Poincar\'e pairing on $H=H^*(X;\mathbb{C})$. For simplicity of the exposition, we will assume that all non-zero cohomology classes are of even degree only; otherwise one has to view $H$ as a vector superspace (see \cite{KM}). The genus-$0$ GW invariants of $X$ allow one to endow $H$ with a commutative associative product $\bullet_t$ parameterized by $t\in H$, known as the \emph{quantum cup product} \cite{KM, RT}.

Assuming the basis $\{v_i\}_{i=1}^N$ of $H$ is homogeneous, we also introduce the {\em Euler vector field} on $H$:
\begin{equation*}
E=\sum_{i=1}^N (1-d_i)t^i \frac{\partial}{\partial t^i} + \sum_{i=1}^N \rho^i \frac{\partial}{\partial t^i} \,,
\end{equation*}
where $d_i = (1/2) \deg v_i$, and $\rho^i$ and $t^i$ are the coordinates respectively of $c_1(TX)$ and $t$ relative to the basis $\{v_i\}_{i=1}^N$. The Poincar\'e pairing and the quantum multiplication $\bullet_t$ are homogeneous of degrees respectively $2-D$ and $1$ with respect to $E$, where $D={\rm dim}_\mathbb{C}\ X$. 

One of the key facts in GW theory is that the following system of differential equations is compatible:
\begin{align}
\label{frob_eq1}
z \d_{t^i}  J(t,z) & =  v_i\bullet_t J(t,z)\,, \qquad\quad 1\leq i\leq N \,, \\
\label{frob_eq2}
(z\d_z+E) J(t,z) &= \theta J(t,z) \,,
\end{align}
where $\theta$ is the {\em Hodge grading operator} defined by $\theta(v_i)=(D/2 - d_i) v_i$. 

The quantum multiplication is called {\em semi-simple} if there are local coordinates $u^i$ on $H$, known as {\em canonical coordinates}, in which both the Poincar\'e pairing and the multiplication assume a diagonal form:
\begin{equation*}
\d/\d u^i \, \bullet_t\, \d/\d u^j = \delta_{ij}\d/\d u^j,\qquad 
(\d/\d u^i,\d/\d u^j ) = \delta_{ij}/\Delta_j
\end{equation*}
for some non-zero functions $\Delta_j$. Examples of manifolds with
semi-simple quantum cohomology include Grassmanians and Fano toric
manifolds. It was conjectured by Givental \cite{G1} and proved by
Teleman \cite{T} that if the quantum multiplication is semi-simple, then $\D_X$ is given by a formula of the following type:
\begin{equation}\label{giv_formula}
\D_X(\t) =  \widehat{G}_t \
\prod_{i=1}^N \D_{\pt}(\t^i), 
\end{equation}
where the variables $\t^i$ are the coordinates of $\t$ with respect to the basis $\sqrt{\Delta_i} \, \d/\d u^i$ and $\widehat{G}_t$ is a certain differential operator defined only in terms of the canonical coordinates and certain solutions of the differential equations \eqref{frob_eq1} and \eqref{frob_eq2} (see Sect.\ \ref{sec4} below). Givental's formula \eqref{giv_formula} implies that $\D_X$ can be reconstructed only from genus-0 GW invariants and the higher-genus theory of the point.

Motivated by GW theory, Dubrovin introduced the notion of a {\em Frobenius manifold} (see \cite{Du, Ma}). Locally, this is defined as follows. Let $H$ be a vector space whose tangent spaces $T_tH$ are Frobenius algebras with identity  $1$, i.e., there exist a non-degenerate bilinear pairing $(\cdot,\cdot)_t$ and a commutative associative multiplication $\bullet_t$ such that $(v\bullet_t w_1, w_2)_t= (w_1,v\bullet_t w_2)_t$. Assume also that the pairing is flat and homogeneous (of degree $2-D$) with respect to an Euler vector field $E$. We say that the Frobenius algebras form a {\em Frobenius structure} of conformal dimension $D$ if the system of equations \eqref{frob_eq1}, \eqref{frob_eq2} is compatible. The notion of semi-simplicity still makes sense in such an abstract setting. Therefore, following Givental \cite{G1}, we use formula \eqref{giv_formula} to define the {\em total descendant potential} of the semi-simple Frobenius manifold. 

By the results of Givental \cite{G2}, the Virasoro conjecture holds in
the semi-simple case. 
The construction of integrable hierarchies in the setting of
semi-simple Frobenius manifolds was investigated in \cite{Du, DZ2, DZ5,
  Ge2} using the bi-Hamiltonian formalism. The methods of Dubrovin
and Zhang are quite remarkable. They have recently
confirmed that such an integrable hierarchy exists, provided that a
certain conjecture about polynomiality of the Poisson brackets
holds \cite{DZ5}. This conjecture was partially proved by
Buryak--Posthuma--Shadrin \cite{BPS,BPS2} (the polynomiality of the second
bracket is still an open problem). Another approach is to derive
Hirota's bilinear equations for the tau-function; see \cite{OP, M1, M2, MT1, MT2, G3, GM, FGM}.

\subsection{Spin curves and the generalized Witten conjecture}

%

Recall from \cite{W2,JKV} that the moduli space of \emph{$h$-spin curves} consists 
of Riemann surfaces $C$ equipped with marked points and a line bundle  $L$ together 
with an isomorphism between $L^{\tensor h}$ and the canonical bundle $K_C(D)$, where 
$D$ is a divisor supported at the marked points. Different choices of $D$ parameterize connected components of the moduli space. Witten conjectured \cite{W2} (see also \cite{JKV}) that the total descendant potential for $h$-spin curves  is a $\tau$-function for the $h$-th Gelfand--Dickey hierarchy. 
This function is uniquely characterized as the solution that also satisfies the string equation. 

Witten's conjecture can be formulated also in the language of vertex algebras. Let  $\W_h$ be the Zamolodchikov--Fateev--Lukyanov \emph{$\W$-algebra} (see Sect.\ \ref{sec_2.5} below).
 According to Adler and van Moerbeke \cite{AV} there is a unique $\tau$-function for the $h$-th Gelfand--Dickey hierarchy solving the string equation. This unique solution is characterized as a highest weight vector for the vertex algebra $\W_h$ (see also \cite{G3,BM}).

On the other hand, the space of miniversal deformations of an \emph{$A_{h-1}$-singularity} can be equipped with a semi-simple Frobenius structure (see \cite{ST} and Sect.\ \ref{sec3_1} below). Givental proved that the corresponding total descendant potential \eqref{giv_formula} is a solution of the $h$-th Gelfand--Dickey hierarchy satisfying the string equation \cite{G3}. Therefore, the proof of Witten's conjecture was reduced to verifying that the total descendant potential of $h$-spin invariants coincides with Givental's function. This was done first by Faber--Shadrin--Zvonkine \cite{FSZ} (now there is a more general approach due to Teleman \cite{T}).   

Following a suggestion by Witten, Fan--Jarvis--Ruan \cite{FJR} generalized the notion of $h$-spin invariants. They introduced the moduli space of Riemann surfaces equipped with orbifold  line bundles satisfying certain algebraic relations, corresponding to a certain class of weighted-homogeneous polynomials. In particular, choosing $f(x)=x^h$ reproduces the $h$-spin invariants. If the polynomial has an isolated critical point of type $X_N = A_N, D_N, E_6, E_7$ or $E_8$ (these are the so-called {\em simple singularities}; see Sect.\ \ref{ssing} below) the total descendant potential of FJRW-invariants  coincides with the total descendant potential of the corresponding singularity. 

It was proved by
Frenkel--Givental--Milanov \cite{GM, FGM} that the total descendant potential $\D_{X_N}$ of a simple singularity is a  $\tau$-function for the \emph{Kac--Wakimoto hierarchy} of type $X_N$ in the principal realization (see \cite{KW}). In the present paper, we will show that $\D_{X_N}$ satisfies suitable \emph{$\W$-constraints}.

\subsection{Main result}
The Virasoro algebra is a Lie algebra, but the $\W$-algebras are not because they involve nonlinearities. Instead, they are \emph{vertex algebras} 
(see \cite{Bo, FLM, K2, FB, LL} and Sect.\ \ref{s4} below).
Informally, a vertex algebra is a vector space $V$ endowed with products $a_{(n)}b \in V$ for all $a,b\in V$ and $n\in\Z$.
An important example is the \emph{Heisenberg vertex algebra} (or Fock space) $\F$ associated to any vector space $\lieh$
equipped with a symmetric bilinear form. 
We let $\lieh$ be the Cartan subalgebra of a finite-dimensional simple Lie algebra $\lieg$ of type $X_N$ $(X=A,D,E)$, and denote by $R$ the root system.

Following \cite{FF1, FF2, FKRW}, we introduce the \emph{$\W$-algebra}  $\W_{X_N}$ 
as the subalgebra of $\F$ given by the intersection of the kernels of the so-called \emph{screening operators}
${e^{\al}}_{(0)}$ $(\al\in R)$.
Equivalently, $\W_{X_N}$ is the space of $\lieg$-invariants in the basic representation 
of the affine Kac--Moody algebra $\hat\lieg$, first considered by I.~Frenkel \cite{F}. 
In particular, $\W_{X_N}$ contains certain Casimirs, the first of which corresponds to the Virasoro algebra.
It is also important that $\W_{X_N}$ is invariant under the action of the Weyl group $W$. 

Let $\si\in W$ be a Coxeter element. Then the principal realization of the basic representation of $\hat\lieg$
admits the structure of a 
\emph{$\si$-twisted representation} of $\F$ (see \cite{LW, KKLW, K1} and Sect.\ \ref{twlat} below). 
We show that the total descendant potential $\D_{X_N}$ of a simple singularity of type $X_N$
lies in a certain completion of this representation. When restricted to $\W_{X_N}$, this representation becomes untwisted, and it gives rise to products 
$a_{(n)} \D_{X_N}$ for every $a\in\W_{X_N}$ and $n\in\Z$. Our main result is the following theorem.

\begin{theorem}\label{t1}
The total descendant potential\/ $\D_{X_N}$ $(X=A,D,E)$ of a simple
singularity satisfies the\/ $\W_{X_N}$-constraints\/ $a_{(n)} \D_{X_N} = 0$
for all\/ $a\in\W_{X_N}$, $n\ge0$.
\end{theorem}

Since $\W_{A_N}$ coincides with $\W_{N+1}$, the above constraints were previously known for type $A_N$ (see \cite{AV,G3,BM}).
It was shown by  Adler and van Moerbeke \cite{AV} that the $\W$-constraints determine the formal power series $\D_{A_N}$ uniquely. 
We conjecture that this is true for all simple singularities.
For type $D_N$, we have an explicit form of the $\W$-constraints, so we expect that one can prove the uniqueness directly as in \cite{AV}.
It is conceivable that the $\W$-constraints can be derived from the Kac--Wakimoto hierarchy and the string equation, but we only know how to do this
for type $A_N$ (cf.\ \cite{KS,AV,vandeL}).
It will also be interesting to find a matrix model for $\D_{X_N}$ generalizing the Kontsevich model from \cite{Ko1,AV} (cf.\ \cite{Kos,Kha,DV}).

\comment{We construct the representation of the vertex algebra $V_Q$ in
terms of period integrals, which were already used in \cite{GM, FGM} to prove that $\D_{X_N}$ is a solution of the Kac--Wakimoto hierarchy. 
We provide a natural interpretation of this fact from the point
of view of vertex algebras, and give a new proof of the results of \cite{GM, FGM}. 
In more detail, let $\CC_{X_N} \subset V_Q \otimes V_Q$ be the so-called \emph{coset model}, defined as the commutant of the diagonal action of $\hat\lieg$ 
(see \cite{GKO, K2} and Sect.\ ??? below). By the Sugawara and coset constructions, there is an element $\om\in\CC_{X_N}$ such that $L_n=\om_{(n+1)}$
generate a Virasoro algebra commuting with the diagonal action of $\hat\lieg$. Note that $\CC_{X_N}$, as a subalgebra of the vertex algebra
$V_Q \otimes V_Q$, acts on the tensor product $\D_{X_N} \otimes \D_{X_N}$.

\begin{theorem}\label{t2}
The total descendant potential\/ $\D_{X_N}$ $(X=A,D,E)$ of a simple
singularity satisfies\/ $a_{(n)} (\D_{X_N}\otimes \D_{X_N}) = 0$
for all\/ $a\in\CC_{X_N}$, $n\ge0$. 
\end{theorem}

In particular, $\D_{X_N}$ solves the Kac--Wakimoto hierarchy of type $X_N$ in the principal realization, which corresponds to the equation 
given by $a=\om$, $n=1$.}

One may try to define the vertex algebra $\W_{X_N}$ for any isolated singularity by taking $R$ to be the set of \emph{vanishing cycles}
(see Sect.\ \ref{mfib} below). 
It is easy to see that the Virasoro vertex algebra $\W_2$ is always contained in $\W_{X_N}$. 
One of the problems, however, is to determine whether $\W_{X_N}$ is larger than $\W_2$, and to suitably modify the definition of $\W_{X_N}$ so that it is. 
This will be pursued in a future work.

Let $M$ be a twisted module over a vertex algebra $V$ (see \cite{FFR, D} and Sect.\ \ref{twrep} below).
Then for every $a\in V$, there is
a formal power series $Y(a,\la)$ whose coefficients are linear operators on $M$.
The main idea of the present paper is to construct globally defined operator-valued functions
$X(a,\la)$ whose Laurent series expansions at $\la=\infty$ coincide with $Y(a,\la)$.
They have the form
\beq\label{gf}
X(a,\la) = \sum_{K} I^{(K)}_a(\la)\, e_K \,, \qquad a\in V \,,
\eeq
where $\{e_K\}$ is some (graded) basis of $\End M$ and the
coefficients  $I^{(K)}_a(\la)$ are multivalued analytic functions in
$\la$ on the extended complex plane $\CP^1=\C\cup\{\infty\}$ having
a finite order pole at finitely many points $u_i\in \CP^1$.
The composition of such series and the corresponding operator product expansion
(in the form of Proposition \ref{pnprod} below) make sense only locally near
each singular  point $u_i$, using the formal $(\la-u_i)$-adic
topology. In other words, the $n$-th product of twisted fields a priori
is defined only locally near each singular point. In our case, however,
these local $n$-th products turn out to be global objects: there
is a series of the type \eqref{gf} such that its Laurent series
expansions at each singular point agree with the given ones. 

The above idea is realized here for the Heisenberg vertex algebra $\F$, giving rise to what may be called
a \emph{$W$-twisted representation}. It has the property that the monodromy operator associated to a big loop around $0$ is given by the action of a Coxeter element $\si\in W$,
while the monodromy around the other singular points $\la=u_i$ is given by simple reflections from $W$.
The construction looks very natural, and it would be interesting to find other examples as well. 
It is also interesting to compare our approach to other geometric approaches such as \cite{BD,FS,MSV}.

\subsection{Organization of the paper}
The size of the paper has increased substantially as we tried to
make the text accessible to a wider audience. We have
included several sections with background material, as well as an
extensive list of references (complete only to the best of our
knowledge). 

Section \ref{s4} reviews standard material in the theory of vertex
algebras. The main goal is to introduce the notion of a $\W$-algebra and
to construct explicit elements in the $\W$-algebra (see Proposition \ref{pwalg}, which is probably new). 

In Section \ref{twmod}, we give background material on twisted
representations of vertex algebras. We prove a formula for the operator product expansion
(see Proposition \ref{pnprod}), which may be used instead of
the Borcherds identity in the definition of a twisted
representation. This formula is used later in an essential way in
order to extend analytically the twisted fields. 

In Section \ref{s2}, we introduce
the main object of our study, the Frobenius manifold structure on the
space of miniversal deformations of a germ of a holomorphic function
with an isolated critical point. We also recall two important operator series: the calibration $\S_t$ 
and the formal asymptotical operator $R_t$, which are used to construct
Givental's quantization operator (see \eqref{giv_formula}). Finally, we introduce the
period integrals, which are an important ingredient in our construction. 

In Section \ref{sec4}, we present
Givental's quantization formalism and the definitions of the total
descendant and the total ancestor potentials.
In particular, we recall how the quantized operators $\widehat\S_t$ and $\widehat R_t$ act on formal power series.

Section \ref{savar} contains the construction of the global twisted
operators $X(a,\la)$ for all $a$ in the Heisenberg vertex algebra $\F$
(see \eqref{gf}). For $a\in\lieh\subset\F$, they are defined using period integrals.
All other operators $X(a,\la)$ are obtained from the generating ones (with $a\in\lieh$)
in terms of normally ordered products and propagators, analogously to the Wick formula
from conformal field theory. The operators $X(a,\la)$ possess remarkable properties.
Their monodromy is determined by the action of the Weyl group on $\F$.
Their Laurent series expansions at $\la=\infty$ give a $\si$-twisted representation of $\F$,
while their expansions at the other critical points $\la=u_i$ give twisted representations of certain
subalgebras of $\F$. The operators $X(a,\la)$ also have nice conjugation properties with respect to
$\widehat\S_t$ and $\widehat R_t$.

We leave the proof of the properties of the propagators for the next Section
\ref{acpf}. There we show that the Laurent expansions of the propagators
near $\infty$ and near the critical points $\la=u_i$ 
agree, i.e., they can be obtained from each other by means
of analytic continuation. This is precisely the place where we
have to use that the singularity is simple. Our argument relies on
the fact that the monodromy group is a finite reflection group and
is a quotient of the Artin--Brieskorn braid group by the normal subgroup
generated by the squares of the generators (see Lemma \ref{path-independence}). 

In Section \ref{s5} we prove \thref{t1}. We first express the $\W$-constraints as the condition that the Laurent expansions
of $X(a,\la)\D_{X_N}$ at $\la=\infty$ have no negative powers of $\la$ for $a\in\W_{X_N} \subset\F$. 
We deduce this from the regularity at each of the critical points $\la=u_i$, where the statement reduces
to the case of an $A_1$-singularity, due to the properties of $X(a,\la)$. Then the $\W$-constraints for $\D_{X_N}$ 
are reduced to a verification of the Virasoro constraints for $\D_{A_1}=\D_\pt$, which are known to be true.
\comment{
 In a similar way, in Section ???, we obtain \thref{t2} by applying our construction to the twisted
representation of $V_Q\otimes V_Q$ on $M\otimes M$.
It is interesting that in this way we obtain more equations than the standard Hirota bilinear equations of the Kac--Wakimoto hierarchy \cite{KW}.
However, one can show that these additional equations are consequences of the standard ones. }

\section{Vertex algebras and $\W$-algebras}\label{s4}

The notion of a vertex algebra introduced by Borcherds \cite{Bo}
provides a rigorous algebraic description of two-dimensional 
chiral conformal field theory (see e.g.\ \cite{BPZ, Go, DMS}).
In this section, we briefly recall the definition and several important examples; 
for more details, see \cite{FLM, K2, FB, LL}. 

\subsection{Affine Lie algebras}\label{aff}

Let $\lieg$ be a finite-dimensional Lie algebra equipped
with a symmetric invariant bilinear form $(\cdot|\cdot)$,
normalized so that the square length of a long root is $2$
in the case when $\lieg$ is simple. For $\lieg=\mathfrak{sl}_{N+1}$
this gives $(a|b)=\tr(ab)$. The \emph{affine Lie algebra}
$\hat\lieg = \lieg[t,t^{-1}] \oplus \C K$
has the Lie brackets
\begin{equation}\label{aff1}
[a_m,b_n] = [a,b]_{m+n} + m \delta_{m,-n} (a|b) K \,, \qquad
a_m=at^m \,,
\end{equation}
and $K$ is central (see \cite{K1}).

Recall that the generalized \emph{Verma module} 
$M(\Lambda_0) = \Ind^{\hat\lieg}_{\lieg[t]\oplus\C K} \C$
is defined by letting $\lieg[t]$ act trivially on $\C$
and $K$ act as $1$. Then $K$ acts as the identity on
the whole $M(\Lambda_0)$, so the level is $1$.
The unique irreducible quotient $V(\Lambda_0)$ of $M(\Lambda_0)$  
is known as the \emph{basic representation} (see \cite{K1}).
Both $V(\Lambda_0)$ and $M(\Lambda_0)$ are highest weight
representations with highest weight vectors the image of
$1\in\C$, which we will denote by $\vac$. Moreover,
due to \cite{FZ}, they both have the structure of a vertex algebra.

\subsection{Vertex algebras}\label{vert}

A \emph{vertex algebra} \cite{Bo, FLM, K2, FB, LL} is a vector space $V$ 
(space of states) with a distinguished vector $\vac\in V$ 
(vacuum vector), together with a linear map 
(state-field correspondence)
\begin{equation}\label{vert2}
Y(\cdot,\ze)\cdot \colon V \otimes V \to V(\!(\ze)\!) := V[[\ze]][\ze^{-1}] \,.
\end{equation}
Thus, for every state $a\in V$, we have the field
$Y(a,\ze) \colon V \to V(\!(\ze)\!)$. This field can be viewed as
a formal power series from $(\End V)[[\ze,\ze^{-1}]]$, which 
involves only finitely many negative powers of $\ze$ when
applied to any vector.
The coefficients
in front of powers of $\ze$ in this expansion are known as the
\emph{modes} of $a$:
\begin{equation}\label{vert4}
Y(a,\ze) = \sum_{n\in\Z} a_{(n)} \, \ze^{-n-1} \,, \qquad
a_{(n)} \in \End V \,.
\end{equation}
As usual, the \emph{formal residue} $\Res_\ze$ of a formal power series is defined as the coefficient of $\ze^{-1}$.
Then 
\begin{equation}\label{fres}
a_{(n)} = \Res_\ze \ze^n Y(a,\ze) \,.
\end{equation}

The vacuum vector $\vac$ plays the role of an identity in the sense that
\begin{equation*}
a_{(-1)}\vac = \vac_{(-1)} a = a \,, \qquad 
a_{(n)}\vac = 0 \,, \quad n\geq 0 \,.
\end{equation*}
In particular, $Y(a,\ze)\vac \in V[[\ze]]$ is regular at $\ze=0$, and
its value at $\ze=0$ is equal to $a$.
The main axiom for a vertex algebra is the following \emph{Borcherds identity}
(also called Jacobi identity \cite{FLM})
satisfied by the modes:
\begin{equation}\label{vert5}
\begin{split}
\sum_{j=0}^\infty (-1)^j & \binom{n}{j} 
\Bigl( 
a_{(m+n-j)}(b_{(k+j)}c)
- (-1)^n \, b_{(k+n-j)}(a_{(m+j)}c)
\Bigr)
\\ 
&=
\sum_{j=0}^\infty \binom{m}{j} (a_{(n+j)}b)_{(k+m-j)}c \,,
\end{split}
\end{equation}
where $a,b,c \in V$. Observe that the above sums are finite, because
$a_{(n)}b = 0$ for sufficiently large $n$.

We say that a vertex algebra $V$ is (strongly) \emph{generated} by a subset $S\subset V$ if $V$
is linearly spanned by the vacuum $\vac$ and all elements of the form
\begin{equation*}
\quad {a_1}_{(n_1)} \cdots {a_r}_{(n_r)} \vac \,, \qquad \text{where} \quad r\geq1 \,, \; a_i \in S \,, \; n_i<0 \,.
\end{equation*}

\subsection{Lattice vertex algebras}\label{lat}
Let $Q$ be an (even) integral lattice with a symmetric
nondegenerate bilinear form $(\cdot|\cdot)$. 
We denote by   
$\lieh = \C\otimes_\Z Q$ the corresponding complex 
vector space considered as an abelian Lie algebra.
The affine Lie algebra  
$\hat\lieh = \lieh[t,t^{-1}] \oplus \C K$ 
is called the \emph{Heisenberg algebra}.
Its irreducible highest weight representation
\begin{equation*}
\F := M(\Lambda_0) \cong S(\lieh[t^{-1}]t^{-1})
\end{equation*}
is known as the (bosonic) \emph{Fock space}.

Following \cite{FK, Bo}, we consider a certain
$2$-cocycle $\ep\colon Q \times Q \to \{\pm1\}$ 
such that
\begin{equation}\label{lat2}
\ep(\al,\al) = (-1)^{|\al|^2(|\al|^2+1)/2}  \,, \quad
|\al|^2:=(\al|\al) \,,
\qquad \al\in Q \,,
\end{equation}
and the algebra $\C_\ep[Q]$ with a basis
$\{ e^\al \}_{\al\in Q}$ and multiplication
\begin{equation*}
e^\al e^\be = \ep(\al,\be) e^{\al+\be} \,.
\end{equation*}
Such a $2$-cocycle $\ep$ is unique up to equivalence, and can be chosen
to be bimultiplicative.

The \emph{lattice vertex algebra} \cite{Bo, FLM, K2, FB, LL}
is the tensor product $V_Q=\F\otimes\C_\ep[Q]$,
where the vacuum vector is $\vac\otimes e^0$.
We let the Heisenberg algebra act on $V_Q$ so that
\begin{equation*}
a_n e^\be = \delta_{n,0} (a|\be) e^\be \,, \quad n\geq0 \,, \qquad
a\in\lieh \,, \; a_n=at^n \,.
\end{equation*}
%
The state-field correspondence on $V_Q$ is uniquely determined by the
generating fields:
\begin{align}\label{lat4}
& Y(a_{-1}\vac,\ze) = \sum_{n\in\Z} a_n \, \ze^{-n-1} \,, \qquad a\in\lieh \,,
\\ \label{lat5}
& Y(e^\al,\ze) = e^\al \ze^{\al_0} 
\exp\Bigl( \sum_{n<0} \al_n \frac{\ze^{-n}}{-n} \Bigr) 
\exp\Bigl( \sum_{n>0} \al_n \frac{\ze^{-n}}{-n} \Bigr) \,,
\end{align}
where $\ze^{\al_0} e^\be = \ze^{(\al|\be)} e^\be$.

Notice that $\F\subset V_Q$ is a vertex subalgebra, which we call the \emph{Heisenberg vertex algebra}.
The map $\lieh\to\F$ given by $a\mapsto a_{-1}\vac$ is injective.
{}From now on, we will slightly abuse the notation and identify
$a\in\lieh$ with $a_{-1}\vac \in\F$; then $a_{(n)}=a_n$ for all $n\in\Z$.


\subsection{The vertex algebra $\W_{X_N}$}\label{sec_2.5}

Let $\lieg$ be a finite-dimensional simple Lie algebra of type $X_N$ $(X=A,D,E)$.
We denote by $R$ and $Q$ the set of roots and the root lattice, respectively.
Following \cite{FF2}, we define $\W_{X_N}$ 
as the intersection of the Fock space $\F\subset V_Q$ and the kernels of all \emph{screening operators}
\begin{equation*}
{e^{\al}}_{(0)} = \Res_\ze Y(e^{\al},\ze) \,, \qquad \al\in R \,.
\end{equation*}
In a vertex algebra, any zero mode acts as a derivation
of the products $a_{(n)}b$ defined by \eqref{vert4}, and the kernel
of a derivation is a vertex subalgebra (see e.g.\ \cite{K2}).
Thus, $\W_{X_N} \subset \F$ is a vertex subalgebra.
The vertex algebras $\W_{X_N}$ are 
examples of non-linear extensions of the 
Virasoro algebra known as \emph{$\W$-algebras} 
(see \cite{Z, FL, BS, FF1, FF2, FKW, FKRW, FB, DK1} and the references therein).

The algebra $\W_{X_N}$ contains the \emph{Virasoro element}
\begin{equation*}
\om = \frac12\sum_{i=1}^N v^i_{(-1)} v_i  \in\W_{X_N} \subset \F \,,
\end{equation*}
where $\{v_i\}$ and $\{v^i\}$ are bases of $\lieh$ dual with respect to $(\cdot|\cdot)$.
The modes $L_n = \om_{(n+1)}$ satisfy the commutation relations of the Virasoro algebra with central charge $N$ 
(see e.g.\ \cite{K2}).
The operator $L_0$ provides a grading of $V_Q$ such that
$\deg a_n = -n$ and $\deg e^\al = |\al|^2/2$ for $a\in\lieh$, $\al\in Q$.

It was proved in \cite{FF2, FKRW} that the vertex algebra $\W_{X_N}$
is freely generated by $N$ elements of degrees $m_1+1,\dots,m_N+1$,
where $m_k$ are the exponents of $\lieg$. This means that
$\W_{X_N}$ has a PBW-type basis (see \cite{DK1} for more on freely generated vertex algebras).
For convenience, the exponents of type $ADE$ are listed in Table \ref{table1} in Sect.\ \ref{ssing} below.

\begin{remark}\label{rwn}
In the case $\lieg=\liesl_{N+1}$, the vertex algebra $\W_{A_N}$ coincides
with the Zamolodchikov--Fateev--Lukyanov algebra $\W_{N+1}$ introduced in
\cite{Z, FL}, for central charge $N$.
In particular, $\W_{A_1} = \W_2$ is the
Virasoro vertex algebra with central charge $1$ (cf.\ \cite{FZ}).
\end{remark}

The Frenkel--Kac construction of the basic representation $V(\Lambda_0)$
can be interpreted as an isomorphism of vertex algebras $V(\Lambda_0) \cong V_Q$
(see \cite{FK, Se, K1, K2}). The Lie algebra $\lieg$ is realized in $V_Q$
as the zero modes $a_{(0)} = a_0$ for $a\in\lieh$ 
and ${e^\al}_{(0)}$ for $\al\in R$.
Hence,
$\W_{X_N}$ can be identified with the space 
of $\lieg$-invariants in the basic representation 
of $\hat\lieg$, first considered by I.\ Frenkel \cite{F}.
In particular, all elements of $\W_{X_N}$ are fixed by the Weyl group $W$ of~$\lieg$. 

\begin{example}\label{ewalg1}
For every $d \geq 1$ we have \cite{F}:
\begin{equation}\label{fr1}
\om^d := \sum_{i=1}^N v^i_{(-d)} v_i - \sum_{\al\in R} {e^\al}_{(-d)} e^{-\al}
\in\W_{X_N} \,.
\end{equation}
Note that $\deg\om^d=d+1$, and $\om^1$ is a scalar multiple of $\om$ by the Sugawara construction
(see e.g.\ \cite{K2}). 
\end{example}

Another way to construct elements of $\W_{X_N}$ is provided by the next result, which seems new.

\begin{proposition}\label{pwalg}
Suppose that\/ $\la_0\in\lieh$ is such that\/ $(\la_0|\al)=0,\pm1$ for all\/ $\al\in R$. Then
\begin{equation}\label{fr5}
\nu^d := \sum_{\la\in W\la_0} {e^{\la}}_{(-d)} e^{-\la}
\in\W_{X_N} \,.
\end{equation}
\end{proposition}
%
%
\begin{proof}
It suffices to check that ${e^\al}_{(0)} \nu^d = 0$ for all $\al\in R$. Since ${e^\al}_{(0)}$ is a derivation,
we have
\begin{equation*}
{e^\al}_{(0)} \nu^d = \sum_{\la\in W\la_0} ({e^\al}_{(0)} e^{\la})_{(-d)} e^{-\la}
+ \sum_{\la\in W\la_0} {e^{\la}}_{(-d)} ({e^\al}_{(0)} e^{-\la}) \,.
\end{equation*}
By \eqref{lat5},
\begin{equation*}
{e^\al}_{(0)} e^\la
= \Res_\ze Y(e^\al,\ze) e^\la 
= \Res_\ze \ep(\al,\la) \ze^{(\al|\la)} \exp\Bigl( \sum_{n<0} \al_n \frac{\ze^{-n}}{-n} \Bigr) e^{\la+\al}
\end{equation*}
is zero for $(\al|\la) \geq 0$ and is equal to $\ep(\al,\la) e^{\la+\al}$ when $(\al|\la) =-1$. 
Now if $(\al|\la) =-1$ for some $\la=w\la_0$, then $(\al|-r_\al\la) = (\al|-\la-\al) = -1$ as well.
Due to \eqref{lat2} and bimultiplicativity,
\begin{equation*}
\ep(\al,-r_\al\la) = \ep(\al,\la+\al) = \ep(\al,\la)\ep(\al,\al) = -\ep(\al,\la) \,.
\end{equation*}
Therefore, the terms with ${e^{\la}}_{(-d)} e^{-\la}$ and ${e^{r_\al\la}}_{(-d)} e^{-r_\al\la}$ cancel.
\end{proof}

\begin{example}\label{ewalg2}
The roots of type $A_N$ can be realized as $v_i-v_j$ where $\{v_i\}$ is an orthonormal basis for $\R^{N+1}$.
The Weyl group $W=S_{N+1}$ acts by permutations on $v_1,\dots,v_{N+1}$.
Then $\la_0=v_1$ satisfies the condition
of \prref{pwalg}, and
\begin{equation*}
\nu^d = \sum_{i=1}^{N+1} {e^{v_i}}_{(-d)} e^{-v_i} \mod (v_1+\cdots+v_{N+1}) \in\W_{A_N} \,.
\end{equation*}
Note that the fields $Y(e^{\pm v_i},\ze)$ are the so-called charged free fermions (see e.g.\ \cite{K2}).
It follows from the results of \cite{F,FKRW} that the elements $\nu^d$ ($1\leq d\leq N$) generate $\W_{A_N}$.
\end{example}

\begin{example}\label{ewalg3}
The roots of type $D_N$ can be realized as $\pm v_i \pm v_j$, where $\{v_i\}$ is an orthonormal basis
for $\R^N$. The Weyl group acts on $v_1,\dots,v_{N}$ by permutations and an even number of sign changes.
Then $\la_0=v_1$ works and
\begin{equation*}
\nu^d = \sum_{i=1}^N {e^{v_i}}_{(-d)} e^{-v_i} + \sum_{i=1}^N {e^{-v_i}}_{(-d)} e^{v_i} 
\in \W_{D_N} \,.
\end{equation*}
It is not hard to check that we also have
\begin{equation*}
\pi^N := {v_1}_{(-1)} \cdots {v_{N-1}}_{(-1)} v_N \in \W_{D_N} \,.
\end{equation*}
Due to Theorem 14.2 in \cite{KWY}, $\W_{D_N} $ is generated by $\pi^N$ and $\nu^d$ $(d\geq1)$.
\end{example}

\begin{example}\label{ewalg4}
The root system of type $E_N$ $(N=6,7)$ can be realized in terms of an orthonormal basis $\{v_i\}$ for $\R^{N+1}$
(see e.g.\ \cite{K1}, Chapter 6). Then $\la_0=v_1+v_2$ satisfies the condition of \prref{pwalg}.
In the case of $E_8$ such an element $\la_0$ does not exist.
\end{example}

\section{Twisted representations of vertex algebras}\label{twmod}

In this section, we review the notion of a twisted representation of a
vertex algebra, and we derive several properties of twisted representations.
We also discuss twisted representations of the $\W$-algebras $\W_{X_N}$.

\subsection{Definition of twisted representation}\label{twrep}
Let $V$ be a vertex algebra, as in Sect.\ \ref{vert}.
A \emph{representation} (or \emph{module}) of $V$ is a vector space $M$ endowed with a
linear map $Y(\cdot,\ze)\cdot \colon V \otimes M \to M(\!(\ze)\!)$
(cf.\ \eqref{vert2}, \eqref{vert4}) such that the Borcherds identity
\eqref{vert5} holds for $a,b\in V$, $c\in M$ (see \cite{FB, LL}).

Now let $\si$ be an automorphism of $V$ of a finite order $h$. Then $\si$ is diagonalizable.
In the definition of a \emph{$\si$-twisted representation} $M$ of $V$ \cite{FFR, D}, the image of the
above map $Y$ is allowed to have nonintegral (rational) powers of $\ze$.
More precisely,
\begin{equation}\label{twlat1}
Y(a,\ze) = \sum_{n\in p+\Z} a_{(n)} \, \ze^{-n-1} \,, \qquad
\text{if} \quad \si a = e^{-2\pi\ii p} a \,, \; p\in\frac1h\Z \,,
\end{equation}
where $a_{(n)} \in \End M$.
Equivalently, the monodromy around $\ze=0$ is given by the action of $\si$:
\begin{equation}\label{twrep2}
Y(\si a,\ze) = Y(a, e^{2\pi\ii}\ze) \,, \qquad a\in V \,.
\end{equation}
The Borcherds identity \eqref{vert5} satisfied by the modes remains the
same in the twisted case.

The above notion of a twisted representation axiomatizes the properties of
the so-called ``twisted vertex operators'' \cite{KP,Le1,FLM1,Le2}, which were used
in the construction of the ``moonshine module'' vertex algebra in \cite{FLM}.
When restricted to the $\si$-invariant subalgebra $V^\si \subset V$ 
(known as an \emph{orbifold}; see \cite{DVVV, KT, DLM} among many other works),
a $\si$-twisted representation for $V$ becomes untwisted for $V^\si$.

\comment{
\begin{remark}
The untwisted representations of the lattice vertex algebra $V_Q$ provide the
Frenkel--Kac construction of level $1$ representations of $\hat\lieg$ 
in the homogeneous realization (see Sect.\ \ref{sec_2.5} and \cite{FK, Se, K1, K2}). 
%
When $\si$ is a Coxeter element, from the $\si$-twisted representations of $V_Q$
one obtains the level 
$1$ representations of $\hat\lieg$ in the principal realization
(see \cite{LW, KKLW, K1}).
\end{remark}
}

\subsection{Consequences of the Borcherds identity}\label{sborid}
For a rational function $f(\ze_1,\ze_2)$ with poles 
only at $\ze_1=0$, $\ze_2=0$ or $\ze_1=\ze_2$, we denote by 
$\io_{\ze_1,\ze_2}$ (respectively, $\io_{\ze_2,\ze_1}$) its expansion
in the domain $|\ze_1| > |\ze_2| > 0$ (respectively, $|\ze_2| > |\ze_1| > 0$).
Explicitly, we have
\begin{equation}\label{iota}
\begin{split}
& \io_{\ze_1,\ze_2} \ze_{12}^n 
= \sum_{j=0}^\infty \binom{n}{j} \ze_1^{n-j} (-\ze_2)^j
\,, \\
& \io_{\ze_2,\ze_1} \ze_{12}^n 
= \sum_{j=0}^\infty \binom{n}{j} \ze_1^j (-\ze_2)^{n-j} \,,
\qquad \text{where} \quad \ze_{12} = \ze_1-\ze_2 \,.
\end{split}
\end{equation}
In particular,
\begin{equation}\label{fde}
\de(\ze_1,\ze_2) 
:= (\io_{\ze_1,\ze_2}- \io_{\ze_2,\ze_1}) \ze_{12}^{-1}
= \sum_{j\in\Z} \ze_1^{-j-1} \ze_2^j
\end{equation}
is the formal \emph{delta-function} (see e.g. \cite{K2,LL}).

The Borcherds identity \eqref{vert5} can be stated equivalently
as follows (see \cite{FFR, D, DL2, BK1}). 

\begin{lemma}\label{lborid}
The Borcherds identity \eqref{vert5} for a\/ $\si$-twisted representation $M$ of a vertex algebra $V$
is equivalent to{\rm:}
\begin{equation}\label{twbor}
\begin{split}
\Res_{\ze_{12}} & Y(Y(a,\ze_{12})b,\ze_2)c \; \io_{\ze_2,\ze_{12}} f(\ze_1,\ze_2) \ze_1^{p}
\\
=& \Res_{\ze_1} Y(a,\ze_1)Y(b,\ze_2)c \; \io_{\ze_1,\ze_2} f(\ze_1,\ze_2) \ze_1^{p}
\\
&- \Res_{\ze_1} Y(b,\ze_2)Y(a,\ze_1)c \; \io_{\ze_2,\ze_1} f(\ze_1,\ze_2) \ze_1^{p}
\end{split}
\end{equation}
for\/ $a,b\in V$, $c\in M$ such that\/ $\si a = e^{-2\pi\ii p} a$,
and every rational function\/ $f(\ze_1,\ze_2)$ with poles 
only at\/ $\ze_1=0$, $\ze_2=0$ or\/ $\ze_1=\ze_2$.
\end{lemma}
Assume that 
$\si a = e^{-2\pi\ii p} a$ and $\si b = e^{-2\pi\ii q} b$ with $p,q\in\frac1h\Z$.
Let $N_{ab}$ be a non-negative integer such that $a_{(n)}b=0$ for all $n \geq N_{ab}$. 
Then setting $f(\ze_1,\ze_2) = \ze_1^{m'} \ze_{12}^{N_{ab}}$
in \eqref{twbor} for all $m'\in\Z$, we obtain the \emph{locality} property \cite{DL1,Li}
\begin{equation}\label{locpr}
\ze_{12}^{N_{ab}} \, Y(a,\ze_1) Y(b,\ze_2) = \ze_{12}^{N_{ab}} \, Y(b,\ze_2) Y(a,\ze_1) \,.
\end{equation}
An important consequence of \eqref{locpr} is that for every $c\in M$
\begin{equation*}
\ze_{12}^{N_{ab}} \, Y(a,\ze_1) Y(b,\ze_2) c \in \ze_1^{-p} \ze_2^{-q} M(\!(\ze_1,\ze_2)\!) \,.
\end{equation*}
The elements of this space have the powers of both $\ze_1$ and $\ze_2$ bounded from below.
Therefore, it makes sense to set $\ze_1=\ze_2$ in such a series, and the result is an element
of $\ze_2^{-p-q} M(\!(\ze_2)\!)$. The same is true if we first differentiate the series.

\begin{proposition}\label{pnprod}
Let\/ $V$ be a vertex algebra, $\si$ an automorphism of\/ $V$, and\/ $M$ a $\si$-twisted representation of\/ $V$.
Then
\begin{equation}\label{locpr3}
\frac1{k!} \d_{\ze_1}^k \Bigl( \ze_{12}^{N} \, Y(a,\ze_1) Y(b,\ze_2) c \Bigr)\Big|_{\ze_1=\ze_2}
= Y(a_{(N-1-k)} b, \ze_2) c
\end{equation}
for all\/ $a,b\in V$, $c\in M$, $k\geq0$, and sufficiently large\/ $N$, where\/ $\ze_{12} = \ze_1-\ze_2$.
We can take\/ $N=N_{ab}$ where\/ $N_{ab}$ is such that\/ \eqref{locpr} holds.
\end{proposition}
\begin{proof}
Without loss of generality, we can suppose again that $\si a = e^{-2\pi\ii p} a$ with $p\in\frac1h\Z$;
then \eqref{twlat1} holds.
Using properties of the formal delta function \eqref{fde} and the formal residue \eqref{fres}, we find
that the left-hand side of \eqref{locpr3} is equal to
\begin{align*}
L:&= \frac1{k!} \Res_{\ze_1} \ze_1^{p} \ze_2^{-p} \de(\ze_1,\ze_2) \, \d_{\ze_1}^{k} \bigl( \ze_{12}^{N_{ab}} \, Y(a,\ze_1) Y(b,\ze_2) c \bigr)
\\
&= \frac1{k!} \Res_{\ze_1} (-\d_{\ze_1})^{k} \bigl(\ze_1^{p} \ze_2^{-p} \de(\ze_1,\ze_2) \bigr) \, \ze_{12}^{N_{ab}} \, Y(a,\ze_1) Y(b,\ze_2) c \,.
\end{align*}
By the Leibniz rule and \eqref{fde}, we have
\begin{equation*}
\frac1{k!} (-\d_{\ze_1})^{k} \bigl(\ze_1^{p} \ze_2^{-p} \de(\ze_1,\ze_2) \bigr)
= \sum_{i=0}^k (-1)^{i} \binom{p}{i} \ze_1^{p-i} \ze_2^{-p} \, (\io_{\ze_1,\ze_2}- \io_{\ze_2,\ze_1}) \ze_{12}^{-1-k+i} \,.
\end{equation*}
Then Borcherds identity \eqref{twbor}, combined with locality \eqref{locpr}, gives that
\begin{equation*}
L=\sum_{i=0}^k (-1)^{i} \binom{p}{i} \Res_{\ze_{12}} Y(Y(a,\ze_{12})b,\ze_2)c \; \io_{\ze_2,\ze_{12}} \ze_1^{p-i} \ze_2^{-p} \ze_{12}^{N_{ab}-1-k+i} \,.
\end{equation*}
Writing explicitly the expansion of $\ze_1=\ze_2+\ze_{12}$ as in \eqref{iota}, and using \eqref{fres}, we obtain
\begin{equation*}
L=\sum_{i=0}^k \sum_{j=0}^\infty (-1)^{i} \binom{p}{i} \binom{p-i}{j} \ze_2^{-i-j} \, Y(Y(a_{(N_{ab}-1-k+i+j)} b,\ze_2)c \,.
\end{equation*}
Notice that the sum over $j$ can be truncated at $j=k-i$, because $a_{(n)}b=0$ for $n \geq N_{ab}$.
Setting $m=i+j$, we get
\begin{equation*}
L=\sum_{m=0}^k \sum_{i=0}^m (-1)^{i} \binom{p}{i} \binom{p-i}{m-i} \ze_2^{-m} \, Y(Y(a_{(N_{ab}-1-k+m)} b,\ze_2)c \,.
\end{equation*}
Now observe that
\begin{equation*}
\sum_{i=0}^m (-1)^{i} \binom{p}{i} \binom{p-i}{m-i}
= \binom{p}{m} \sum_{i=0}^m (-1)^{i} \binom{m}{i}
= \de_{m,0} \,,
\end{equation*}
completing the proof.
\end{proof}

\begin{remark}\label{rnprod}
By reversing the above proof, one can show that, conversely, the product identity \eqref{locpr3} and locality \eqref{locpr}  imply the Borcherds identity \eqref{twbor}.
Therefore, they can replace the Borcherds identity in the definition of twisted representation.
\end{remark}

\begin{remark}\label{rnprod2}
The above proof simplifies significantly in the case of an untwisted representation $M$, as then $p=0$.
In the untwisted case, formula \eqref{locpr3} first appeared in \cite{BN} and \cite{BK2} for vertex algebras and generalized
vertex algebras, respectively. 
It provides a rigorous interpretation of the \emph{operator product expansion} in conformal field theory
(cf.\ \cite{Go, DMS}).
\end{remark}

The following easy consequence of \eqref{locpr3} will be useful later.

\begin{corollary}\label{clocpr}
Assume that\/ $a,b\in V$ and\/ $c\in M$ are such that\/ $Y(a,\ze)c$ and\/ $Y(b,\ze)c$ have no negative powers of\/ $\ze$.
Then the same is true for all\/ $Y(a_{(k)}b,\ze)c$, $k\in\Z$.
\end{corollary}
\begin{proof}
By locality \eqref{locpr}, the product $\ze_{12}^{N_{ab}} \, Y(a,\ze_1) Y(b,\ze_2)$ has no negative powers of $\ze_1$ and $\ze_2$. Then use \eqref{locpr3}.
\end{proof}

\subsection{Twisted Heisenberg algebra}\label{twheis}
Let $\lieh$ be a finite-dimensional vector space equipped with a symmetric
nondegenerate bilinear form $(\cdot|\cdot)$, as in Sect.\ \ref{lat}.
Then we have the Heisenberg algebra $\hat\lieh$ and its highest weight representation, 
the Fock space $\F$, which has the structure of a vertex algebra.
Every automorphism $\si$ of $\lieh$ preserving the bilinear form induces automorphisms
of $\hat\lieh$ and $\F$, which will be denoted again as $\si$. As before, assume that $\si$
has a finite order $h$.

The action of $\si$ can be extended to $\lieh[t^{1/h},t^{-1/h}] \oplus\C K$ by letting
\begin{equation*}
\si(a t^m) = \si(a) e^{2\pi\ii m}t^m \,, \quad \si(K)=K \,,
\qquad a\in\lieh \,, \; m\in\frac1h\Z \,.
\end{equation*}
The \emph{$\si$-twisted Heisenberg algebra} 
$\hat\lieh_\si$ is defined as the set of all $\si$-invariant elements
(see e.g.\ \cite{KP,Le1,FLM1}).
In other words, $\hat\lieh_\si$ is spanned over $\C$ by $K$ and
the elements $a_m = at^m$ such that $\si a = e^{-2\pi\ii m} a$.
This is a Lie algebra with bracket 
(cf.\ \eqref{aff1})
\begin{equation*}
[a_m,b_n] = m \delta_{m,-n} (a|b) K \,, \qquad a,b\in\lieh \,, \;\; m,n\in \frac1h\Z \,.
\end{equation*}
Let $\hat\lieh_\si^+$ (respectively, $\hat\lieh_\si^-$) be the subalgebra of 
$\hat\lieh_\si$ spanned by all elements $a_m$ with $m\geq0$ 
(respectively, $m<0$). 
Elements of $\hat\lieh_\si^+$ are called \emph{annihilation operators}, 
while elements of $\hat\lieh_\si^-$ \emph{creation operators}.

The \emph{$\si$-twisted Fock space} is defined as the generalized Verma module
\begin{equation}\label{twheis2}
\F_\si := \Ind^{\hat\lieh_\si}_{\hat\lieh_\si^+ \oplus\C K} \C \cong S(\hat\lieh_\si^-) \,,
\end{equation}
where $\hat\lieh_\si^+$ acts on $\C$ trivially and $K$ acts as the identity operator.
It is an irreducible highest weight representation of $\hat\lieh_\si$. Moreover, $\F_\si$
has the structure of a $\si$-twisted representation of the vertex algebra $\F$
(see \cite{FLM,FFR,DL2}). This structure can be described as follows. 
We let $Y(\vac,\ze)$ be the identity operator and
\begin{equation}\label{twheis3}
Y(a,\ze) = \sum_{n\in p+\Z} a_{n} \, \ze^{-n-1} \,, \qquad
a\in\lieh \,, \;\; \si a = e^{-2\pi\ii p} a \,,
\end{equation}
where $p\in\frac1h\Z$ (cf.\ \eqref{twlat1}).
These satisfy the locality property \eqref{locpr} because
\begin{equation}\label{twheis4}
[Y(a,\ze_1), Y(b,\ze_2)] = (a|b) \, \d_{\ze_2} \bigl( \ze_1^{-p} \ze_2^{p} \de(\ze_1,\ze_2) \bigr) \,.
\end{equation}
The action of $Y$ on other elements of $\F$ is then determined by applying several times the product
formula \eqref{locpr3}.


\subsection{Twisted representations of lattice vertex algebras}\label{twlat}

Now let $V=V_Q$ where $Q$ is a root lattice of type $X_N$ ($X=A,D,E$), and let
$\si$ be a \emph{Coxeter element} of the corresponding Weyl group (see e.g.\ \cite{Bour}).
Such an element is a product of simple reflections    
$\si = r_{\al_1} \dotsm r_{\al_N}$ where $\{\al_1,\dots,\al_N\}$ 
is a basis of simple roots and $r_\al(\be) = \be - (\al|\be) \al$.
All Coxeter elements are conjugate
to each other; their order is the \emph{Coxeter number} $h$.
The element $\si$ is diagonalizable on $\lieh$ 
with eigenvalues $e^{\,2\pi\ii m_k/h}$ where $m_k$ are the exponents of $\lieg$
(see Table \ref{table1} in Sect.\ \ref{ssing} below). In particular, 
$\si$ has no fixed points in $\lieh$. 

\begin{example}\label{ecox}
For type $A_N$, one Coxeter element acts as the cyclic permutation 
$v_1 \mapsto v_2 \mapsto\cdots\mapsto v_{N+1} \mapsto v_1$,
in the notation of \exref{ewalg2}.
For type $D_N$, in the notation of \exref{ewalg3}, one Coxeter element acts as
$v_1 \mapsto v_2 \mapsto\cdots\mapsto v_{N-1} \mapsto -v_1$,
$v_N \mapsto -v_N$.
\end{example}

For $\al,\be\in Q$, we define
\begin{equation*}
\ep(\al,\be) = (-1)^{L(\al,\be)} \,, \qquad
L(\al,\be):=((1-\si)^{-1}\al|\be) \,.
\end{equation*}
The bilinear form $L(\cdot,\cdot)$ is known 
in singularity theory as the \emph{Seifert form}, and
is integer valued (see e.g.\ \cite{AGV,Eb}).
The bilinearity of $L$ implies that $\ep$ is bimutiplicative
and so it is a $2$-cocycle. Using the $\si$-invariance of $(\cdot|\cdot)$,
one easily checks that $|\al|^2 = 2L(\al,\al)$, which implies \eqref{lat2}.
Observe that $\ep$ is $\si$-invariant:
\begin{equation*}
\ep(\si\al,\si\be) = \ep(\al,\be) \,, \qquad \al,\be\in Q \,.
\end{equation*}
Then $\si$ can be lifted to an automorphism of $V_Q$ of order $h$:
\begin{equation*}
\si(a_m)=\si(a)_m \,, \quad \si(e^\al)=e^{\si\al} \,,
\qquad a\in\lieh \,, \; \al\in Q \,.
\end{equation*}

Under the above simplifying assumptions, the $\si$-twisted Fock space $\F_\si$ defined in \eqref{twheis2} 
can be endowed with the structure of a $\si$-twisted
representation of $V_Q$ (see \cite{KP,Le1,D,DL2,BK1}).
We define $Y(a,\ze)$ as before (see \eqref{twheis3}), and we let
\begin{equation}\label{twlat12}
Y(e^\al,\ze) = U_\al \, \ze^{ -|\al|^2/2 } \,
{:} \exp\Biggl( \sum_{  n\in\frac1h\Z\setminus\{0\} } \al_n \frac{\ze^{-n}}{-n} \Biggr) {:} \,,
\end{equation}
where $U_\al$ are certain nonzero complex numbers. As usual, the colons denote \emph{normal ordering}, 
which means that we put all annihilation operators ($\al_n$ for $n>0$)
to the right of all creation operators ($\al_n$ for $n<0$).

The scalars $U_\al$ satisfy
\begin{equation*}
U_\al U_\be = \ep(\al,\be) B_{\al,\be}^{-1} \, U_{\al+\be} \,,
\end{equation*}
where
\begin{equation*}
B_{\al,\be} := 
h^{ -(\al|\be) } \prod_{k=1}^{h-1} \bigl(1 - e^{2\pi\ii k/h} \bigr)^{ (\si^k\al|\be) } \,.
\end{equation*}
We will also need that the product $Y(e^\al,\ze_1) Y(e^{-\al},\ze_2)$ 
on $\F_\si$ is given by (see e.g.\ \cite{BK1})
\begin{equation}\label{twlat11}
(-1)^{|\al|^2(|\al|^2+1)/2} \, \ze_1^{-|\al|^2/2} \ze_2^{-|\al|^2/2} \,
\io_{\ze_1,\ze_2} f_{\al}(\ze_1,\ze_2) \, E_{\al}(\ze_1,\ze_2) \,,
\end{equation}
where
\begin{equation*}
f_{\al}(\ze_1,\ze_2) = B_{\al,\al}
\prod_{k=0}^{h-1} \bigl(\ze_1^{1/h} - e^{2\pi\ii k/h} \ze_2^{1/h}\bigr)^{ -(\si^k\al|\al) } \,,
\end{equation*}
and
\begin{equation*}
E_{\al}(\ze_1,\ze_2) =
{:} \exp\Bigl( \sum_{ n\in\frac1h\Z\setminus\{0\} } \frac{\al_n}{n} (\ze_2^{-n} - \ze_1^{-n}) \Bigr) {:} \,.
\end{equation*}

\subsection{Twisted representations of $\W_{X_N}$}\label{twwalg}

We will now use the product formula \eqref{locpr3} to compute the explicit action on $\F_\si$ 
of the elements of $\W_{X_N}$ given by \eqref{fr1} and \eqref{fr5}. Introduce the \emph{Fa\`a di Bruno polynomials}
(see Sect.\ 6A in \cite{Di}):
\begin{equation*}
S_n(\al,\ze) = \frac1{n!} \bigl( \d_\ze + \al(\ze) \bigr)^n \vac \,, \quad
\text{where} \quad \al(\ze) = Y(\al,\ze) \,, \;\; \al\in\lieh \,.
\end{equation*}
More explicitly,
\begin{equation*}
S_n(\al,\ze) =  {:} S_n \Bigl( \al(\ze), \frac1{2!} \d_\ze \al(\ze), \frac1{3!} \d_\ze^2 \al(\ze), \dots \Bigr){:} \,,
\end{equation*}
where
\begin{equation*}
S_n(x_1,x_2,x_3,\dots) =
\sum_{\substack{i_1+2i_2+3i_3+\cdots=n \\ i_s \in\Z_{\geq0}}} \,
\frac{x_1^{i_1}}{i_1!} \frac{x_2^{i_2}}{i_2!} \frac{x_3^{i_3}}{i_3!} \cdots
\end{equation*}
are the \emph{elementary Schur polynomials}.
When acting on the $\si$-twisted Fock space $\F_\si$, the coefficients of $S_n(\al,\ze)$ in front 
of powers of $\ze$ are represented by differential operators.

\begin{lemma}\label{ltwlat}
For every\/ $d\geq1$ and\/ $\al\in\lieh$ such that\/ $|\al|^2 \in\Z$, we have
\begin{equation*}
Y({e^\al}_{(-d)} e^{-\al}, \ze) = (-1)^{|\al|^2(|\al|^2+1)/2}
\sum_{k=0}^{|\al|^2-1+d} c_k^\al \, \ze^{-k} S_{|\al|^2-1+d-k}(\al, \ze)
\end{equation*}
when acting on\/ $\F_\si$, where\/ $c_k^\al$ is the coefficient in front of\/ $(x-1)^k$ in the Taylor expansion of
\begin{equation*}
B_{\al,\al} \, x^{-|\al|^2/2} \,
\prod_{k=1}^{h-1} \bigl(x^{1/h} - e^{2\pi\ii k/h} \bigr)^{ ((1-\si^k)\al|\al) }
\end{equation*}
around $x=1$. In particular, $c_0^\al = 1$ and\/
$c_k^\al = c_k^{-\al} = c_k^{\si\al}$.
\end{lemma}
\begin{proof}
We will apply \eqref{locpr3} for $a=e^\al$, $b=e^{-\al}$ and $c\in\F_\si$.
First, we observe that by \eqref{lat5}
\begin{equation*}
Y(e^\al,\ze) e^{-\al}
= \ep(\al,-\al) \, \ze^{-|\al|^2} \exp\Bigl( \sum_{n<0} \al_n \frac{\ze^{-n}}{-n} \Bigr) \vac \,,
\end{equation*}
so we can take $N_{ab} = |\al|^2$. Then on $\F_\si$
the product $Y(e^\al,\ze_1) Y(e^{-\al},\ze_2)$
is given by \eqref{twlat11}; and
\begin{equation*}
\ze_{12}^{|\al|^2} \io_{\ze_1,\ze_2} f_{\al}(\ze_1,\ze_2)
= B_{\al,\al} \prod_{k=1}^{h-1} \bigl(\ze_1^{1/h} - e^{2\pi\ii k/h} \ze_2^{1/h}\bigr)^{ ((1-\si^k)\al|\al) }
\end{equation*}
is well defined for $\ze_1=\ze_2$. Now the proof follows from the fact that
\begin{equation*}
\frac1{n!} \d_{\ze_1}^n E_{\al}(\ze_1,\ze_2) \big|_{\ze_1=\ze_2} = S_n(\al,\ze_2)
\end{equation*}
(see e.g.\ \cite{Di,K2}).
\end{proof}

\begin{lemma}\label{ltwlat2}
Let\/ $a,b\in\lieh$ be such that\/ $\si a = e^{-2\pi\ii p} a$ with\/ $p\in\frac1h\Z$, $0<p<1$.
Then for every\/ $d\geq1$, we have
\begin{equation*}
Y(a_{(-d)} b, \ze) 
= \frac1{(d-1)!} \, {:} \bigl( \d_\ze^{d-1} a(\ze) \bigr) b(\ze) {:}
- d \, \binom{-p+1}{d+1}\,(a|b)\, \ze^{-d-1}
\end{equation*}
when acting on\/ $\F_\si$, where\/ $a(\ze)=Y(a,\ze)$.
\end{lemma}
\begin{proof}
We will apply \eqref{locpr3} with $N_{ab}=2$.
It follows from \eqref{twheis4}, \eqref{twheis3} and \eqref{fde} that
\begin{equation}\label{twheis41}
a(\ze_1) b(\ze_2) = {:} a(\ze_1) b(\ze_2) {:} + (a|b) \, \d_{\ze_2} \io_{\ze_1,\ze_2} \bigl( \ze_1^{-p} \ze_2^{p} \ze_{12}^{-1} \bigr) \,.
\end{equation}
The rest of the proof is straightforward, using \eqref{locpr3}.
\end{proof}

\section{Singularities: Root systems and Frobenius structures}\label{s2}

The marvelous interrelations between singularities and root systems were uncovered in the works of Klein, Du Val, Brieskorn, Looijenga, Arnold, Slodowy, Saito and others \cite{Br, Lo1, A1, Sl1, S2}. We will review only the material needed for the rest of the paper, referring to \cite{AGV,Eb,He} for more details. Our main goal is to introduce the Frobenius structure on the space of miniversal deformations of a germ of a holomorphic function with an isolated critical point. We also introduce the period integrals, which are an important ingredient in our construction. 

\subsection{Milnor fibration}\label{mfib}

Let $f\colon(\C^{2l+1},0)\rightarrow (\C,0)$ be the germ of a holomorphic function with an isolated critical point of multiplicity $N$. Denote by 
\begin{equation*}
H = \C[[x_0,\ldots,x_{2l}]]/(\d_{x_0}f,\ldots,\d_{x_{2l}}f)
\end{equation*}
the \emph{local algebra} of the critical point; then $\dim H=N$. 

\begin{definition}\label{dmvdef}
A \emph{miniversal deformation} of $f$ is a germ of a holomorphic function $F\colon(\C^N\times \C^{2l+1},0)\to (\C,0)$ satisfying the following two properties:
\begin{enumerate}
\item[(1)]
$F$ is a deformation of $f$, i.e., $F(0,x)=f(x)$.
\item[(2)]
The partial derivatives $\d F/\d t^i$ $(1\leq i\leq N)$ project to a basis in the local algebra 
$$
\O_{\C^N,0}[[x_0,\dots,x_{2l}]]/\langle \d_{x_0}F,\dots,\d_{x_{2l}}F\rangle.
$$
\end{enumerate}
Here we denote by $t=(t^1,\dots,t^N)$ and $x=(x_0,\dots,x_{2l})$ the standard coordinates on $\C^N$ and $\C^{2l+1}$ respectively, and $\O_{\C^N,0}$ is the algebra of germs at $0$ of holomorphic functions on $\C^N.$
\end{definition}

We fix a representative of the holomorphic germ $F$, which we denote again by $F$, with a domain $X$ constructed as follows. Let 
\begin{equation*}
B_\rho^{2l+1}\subset \C^{2l+1} \,, \qquad 
B=B_\eta^N\subset \C^N \,, \qquad 
B_\delta^1\subset \C 
\end{equation*}
be balls with centers at $0$ and radii $\rho,\eta$, and $\delta$, respectively.
We set 
\begin{equation*}
S=B\times B_\delta^1 \subset\C^N\times\C \,, \quad 
X=(B\times B_\rho^{2l+1})\cap \phi^{-1}(S) \subset\C^N\times\C^{2l+1} \,,
\end{equation*}
where
\ben
\phi\colon B\times B_\rho^{2l+1}\to B\times\C \,,
\qquad (t,x)\mapsto (t,F(t,x)) \,.
\een 
This map induces a map $\phi\colon X\to S$ and we denote by $X_s$ or $X_{t,\la}$ the fiber 
\ben
X_s = X_{t,\la} = \{(t,x)\in X \,|\, F(t,x)=\la\} \,,\qquad s=(t,\la)\in S.
\een  
The number $\rho$ is chosen so small that for all $r$, $0<r\leq \rho$, the fiber $X_{0,0}$ intersects transversely the boundary $\d B_r^{2l+1}$ of the ball with radius $r$. Then we choose the numbers $\eta$ and $\delta$ small enough so that for all $s\in S$ the fiber $X_s$ intersects transversely the boundary $\d B_\rho^{2l+1}.$ Finally, we can assume without loss of generality that the critical values of $F$ are contained in a disk $B_{\delta_0}^1$ with radius $\delta_0<1<\delta$.
\begin{figure}[htbp]
\begin{center}
\scalebox{0.55}{\includegraphics{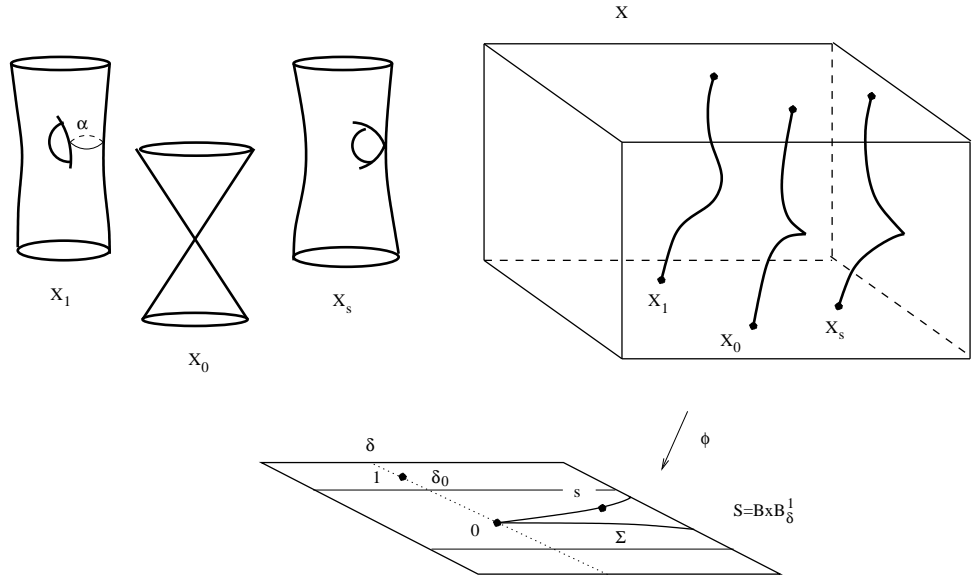}}
\caption{Milnor fibration}
\label{fig:mf}
\end{center}
\end{figure}

Let $\Si$ be the {\em discriminant} of the map $\phi$, i.e., the set
of all points $s\in S$ such that the fiber $X_s$ is singular. Put 
\begin{equation*}
S'=S\setminus{\Si} \subset\C^N\times\C \,, \qquad 
X'=\phi^{-1}(S') \subset X \subset\C^N\times\C^{2l+1} \,.
\end{equation*}
Then the map $\phi\colon X'\to S'$ is a smooth fibration, called 
the \emph{Milnor fibration}. In particular, all smooth fibers are diffeomorphic to $X_{0,1}$.
The middle homology group of the smooth fiber, equipped with the bilinear form
$(\cdot|\cdot)$ equal to $(-1)^l$ times the intersection form, 
is known as the \emph{Milnor lattice} $Q=H_{2l}(X_{0,1};\Z)$. 
    
For a generic point $s\in\Si$, the singularity of the fiber $X_s$
is Morse. Thus, every choice of a path from $(0,1)$ to $s$ avoiding $\Si$
leads to a group homomorphism $Q \to H_{2l}(X_s;\Z)$. The kernel of this
homomorphism is a free $\Z$-module of rank $1$. A generator  
$\al\in Q$ of the kernel is called a \emph{vanishing cycle} if 
$(\al|\al) = 2$. We denote by $R$ the set of all vanishing cycles 
for all possible choices of $s\in\Si$ and paths from $(0,1)$ to $s$.


The fundamental group $\pi_1(S')$ of the base of the Milnor fibration acts on 
the homology of the smooth fiber $Q=H_{2l}(X_{0,1};\Z)$ preserving the
intersection form. The image of $\pi_1(S')$ in $\Aut Q$
will be called the \emph{monodromy group} and denoted by $W$.  
The Picard--Lefschetz
formula tells us that the monodromy associated to a small loop around 
a generic point $s\in\Si$ is given by the reflection $r_\al$, where $\al\in R$ is a cycle vanishing over $s$ and
$r_\al(\be) = \be - (\al|\be) \al$.
Furthermore, $W$ is generated by the reflections $r_\al$ $(\al\in R)$.
The so-called \emph{classical monodromy} $\si\in W$ is the monodromy transformation corresponding to a big loop around~$\Si.$

\subsection{Simple singularities}\label{ssing}

The \emph{simple} singularities are labeled by $ADE$ Dynkin diagrams.
In this case, we can take $f(x)$ to be a polynomial in three variables, as in 
Table \ref{table1}. For further reference, we have also listed the
\emph{Coxeter number} $h$ and the \emph{exponents} $m_1\leq\dots\leq m_N$ of the
corresponding Lie algebra (see e.g.\ \cite{Bour}).


\begin{table}[htb]
\caption{Simple singularities}
\begin{tabular}{cllc}\label{table1}
\textbf{Type} & $\boldsymbol{f}\boldsymbol{(}\boldsymbol{x}\boldsymbol{)}$ & 
\textbf{Exponents} & $\boldsymbol{h}$ \\
\vspace{3pt}
$A_N$ & $x_0^{N+1} \!+\! x_1^2 \!+\! x_2^2$ 
& $1,2,\dots,N$  & $N\!+\!1$ \\ 
\vspace{3pt}
$D_N$ & $x_0^{N-1} \!+\! x_0 x_1^2 \!+\! x_2^2$ 
& $1,3,\dots,2N\!-\!3,N\!-\!1$ & $2N\!-\!2$ \\ 
\vspace{3pt}
$E_6$ & $x_0^4 \!+\! x_1^3 \!+\! x_2^2$ 
& $1,4,5,7,8,11$ & $12$ \\ 
\vspace{3pt}
$E_7$ & $x_0^3 x_1 \!+\! x_1^3 \!+\! x_2^2$ 
& $1,5,7,9,11,13,17$ & $18$ \\ 
\vspace{3pt}
$E_8$ & $x_0^5 \!+\! x_1^3 \!+\! x_2^2$
& $1,7,11,13,17,19,23,29$ & $30$ 
\end{tabular}
\end{table}


For a simple singularity of type $X_N$ $(X=A,D,E)$, the Milnor lattice $Q$
is isomorphic to a root lattice of type $X_N$, while the set $R$ of vanishing
cycles is a root system of type $X_N$. The monodromy group $W$ coincides
with the Weyl group, and the classical monodromy is a 
Coxeter element $\si \in W$.


\subsection{Frobenius structure}\label{sec3_1}
Let $\T_B$ be the sheaf of holomorphic vector fields on $B$. Condition (2) in \deref{dmvdef} implies that the map 
$$
\d/\d{t^i}\mapsto \d F/\d t^i \mod \langle \d_{x_0} F,\dots,\d_{x_{2l}}F\rangle \qquad (1\leq i\leq N)
$$ 
induces an isomorphism between $\T_B$ and $p_*\O_C$, where $p\colon X\to B$ is the natural projection $(t,x)\mapsto t$ and 
\ben
\O_C:=\O_X/\langle \d_{x_0} F,\dots,\d_{x_{2l}}F\rangle
\een
is the structure sheaf of the critical set of $F$. In particular, since $\O_C$ is an algebra, the sheaf $\T_B$ is equipped with an associative commutative multiplication, which will be denoted by $\bullet.$ It induces a product $\bullet_t$ on the tangent space of every point $t\in B$. The class of the function $F$ in $\O_C$ defines a vector field $E\in \T_B$, called the {\em Euler vector field}. 

Given a holomorphic volume form $\omega$ on $(\C^{2l+1},0)$, possibly
depending on $t\in B$, we can equip $p_*\O_C$ with the so-called
\emph{residue pairing}:
\ben
(\psi_1(t,x),\psi_2(t,x)) :=
\Big(\frac{1}{2\pi i}\Big)^{2l+1}\int_{\Gamma_\epsilon} 
\frac{\psi_1(t,x)\ \psi_2(t,x)}
{\d_{x_0} F \cdots \d_{x_{2l}} F } \,\omega\,,
\een
where the integration cycle $\Gamma_\epsilon$ is supported on 
$|\d_{x_0} F|= \cdots =|\d_{x_{2l}}F|=\epsilon$. 
In particular, since $\T_B\cong p_*\O_C$, we get that the residue
pairing induces a non-degenerate complex bilinear form $(\ ,\ )$ on $\T_B$.  

\comment{
The orthogonality conditions mentioned above can be interpreted also in the following way (see \cite{ST}). The sheaf $\H_F$ is equipped with a Gauss--Manin connection:
\ben
\nabla^{\rm G.M.}_{\d/\d t^a} \int e^{F(t,x)/z}g(t,x,z)dx = \int e^{F/z}\Big(z^{-1}\frac{\d F}{\d t^a}g +\frac{\d g}{\d t^a}\Big) dx
\een
and 
\ben
\nabla^{\rm G.M.}_{\d/\d z} \int e^{F(t,x)/z}g(t,x,z)dx = \int e^{F/z}\Big(-z^{-2}Fg +\frac{\d g}{\d z}\Big) dx.
\een}

For $t\in B$ and $z\in\C^*$, let
$\B_{t,z}$ be a semi-infinite cycle in $\C^{2l+1}$ of the following type:
\ben
\B_{t,z}\in \lim_{\rho \to \infty} \, H_{2l+1}(\C^{2l+1},\{ 
  \mathrm{Re}\, z^{-1} F(t,x)<-\rho\} ;\C) \cong \C^N \,.
\een
The above homology groups form a vector bundle on $B\times \C^*$ equipped
naturally with a Gauss--Manin connection, and $\B=\B_{t,z}$ may be viewed as a
flat section. According to K.\ Saito's theory of {\em primitive forms} \cite{SaK,MS}
there exists a form $\omega$, called primitive, such that the oscillatory
integrals ($d^B$ is the de Rham differential on $B$)
\ben
J_\B(t,z):= (2\pi z)^{-l-\frac{1}{2}}\ (zd^B)\, 
\int_{\B_{t,z}} e^{z^{-1}F(t,x)}\omega \in \T_B^*
\een
are horizontal sections for the following connection: 
\beqa\label{frob_eq3}
\nabla_{\d/\d t^i} & = &  \nabla^{\rm L.C.}_{\d/\d t^i} - z^{-1}(\d_{t^i} \bullet_t),
\qquad 1\leq i\leq N \\
\label{frob_eq4}
\nabla_{\d/\d z} & = &  \d_z - z^{-1} \theta + z^{-2} E\bullet_t \,.
\eeqa
Here $\nabla^{\rm L.C.}$ is the Levi--Civita connection associated with the residue pairing and
\ben
\theta:=\nabla^{\rm L.C.}E-\Big(1-\frac{d}{2}\Big){\rm Id},
\een 
where $d$ is some complex number. 

In particular, this means that the residue pairing and the multiplication $\bullet$ form a {\em Frobenius structure} on $B$
of conformal dimension $d$ with identity $1$ and Euler vector field $E$. For the definition of a Frobenius structure we refer to \cite{Du,Ma} (see also Sect.\ \ref{sfrobgiv}). 

\begin{example}\label{eprimsim}
For simple singularities, the standard volume form
\ben
\omega = dx_0\wedge dx_1\wedge \cdots \wedge dx_{2l} 
\een
is the only primitive form, up to a constant factor.
\end{example}

Assume that a primitive form $\omega$ is chosen. Note that the flatness of the Gauss--Manin connection implies that the residue pairing is flat. Denote by $(\tau^1,\dots, \tau^N)$ a coordinate system on $B$ that is flat with respect to the residue metric, and write $\partial_i$ for the vector field $\partial/\partial{\tau^i}$. We can further modify the flat coordinate system so that the Euler field is the sum of a constant and linear fields: 
\ben
E=\sum_{i=1}^N (1-d_i) \tau^i \partial_{i} + \sum_{i=1}^N \rho_i \partial_i \,.
\een
The constant part represents the class of $f$ in $H$, and the spectrum
of degrees $d_1,\dots, d_N$ ranges from $0$ to $d.$ 
Note that in the flat coordinates $\tau^i$ the operator $\theta$ (called sometimes the \emph{Hodge grading operator}) assumes diagonal form:
\ben
\theta(\d_i) = \Bigl(\frac{d}{2}-d_i\Bigr) \d_i \,, \qquad\quad 1\leq i\leq N \,.
\een

\subsection{Period integrals}\label{sec3_3}
Given a middle homology class $\al\in H_{2l}(X_{0,1};\C)$, we denote by $\al_{t,\gl}$ its parallel transport to 
the Milnor fiber $X_{t,\gl}$. Let $d^{-1}\omega$ be any $2l$-form whose 
differential is $\omega$. We can integrate $d^{-1}\omega$ over $\al_{t,\gl}$
and obtain multivalued functions of $\gl$ and $t$ 
ramified around the discriminant in $S$ (over which 
the Milnor fibers become singular).

\begin{definition}\label{dperiods}
To $\al\in\lieh=H_{2l}(X_{0,1};\C)$, we associate the {\em 
period vectors} $I^{(k)}_\al(t,\gl)\in H\ (k\in \Z)$ defined by
\beq\label{periods} 
(I^{(k)}_\al(t,\gl), \partial_i):= 
-(2\pi)^{-l} \d_\gl^{l+k} \d_i \int_{\al_{t,\gl}} d^{-1}\omega \,,
\qquad 1\leq i\leq N \,.
\eeq 
\end{definition}

Note that this definition is consistent with the operation of stabilization of
singularities. Namely, adding the squares of two new variables does not change
the right-hand side, since it is offset by an extra differentiation 
$(2\pi)^{-1}\partial_{\gl}$. In particular, this defines the period
vector for a negative value of $k\geq -l$ with $l$ as large as one wishes.
Note that, by definition, we have 
\ben
\d_\la I^{(k)}_\al(t,\gl) = I^{(k+1)}_\al(t,\gl) \,, \qquad \al\in\lieh \,, \;\; k\in\Z\,.
\een
The following lemma is due to A.\ Givental \cite{G1}.

\begin{lemma}\label{lem:periods}
The period vectors \eqref{periods} satisfy the differential 
equations
\begin{align}\label{periods:de1}
\d_i I^{(k)} &= -\d_i\bullet_t(\d_\gl I^{(k)})\,, \qquad\quad 1\leq i\leq N \,, 
\\ \label{periods:de2}
(\gl-E\bullet_t) \d_{\gl}I^{(k)} &= \Bigl(\theta-k-\frac12\Bigr) I^{(k)} \,.
\end{align}
\end{lemma}

Using equation \eqref{periods:de2},
we analytically extend the period vectors to all 
$|\gl|>\delta$. It follows from \eqref{periods:de1} that the period vectors 
have the symmetry
\beq\label{tinv}
I^{(k)}_\al(t,\gl)\ = \ I^{(k)}_\al(t-\gl\one,0) \,,
\eeq  
where $t \mapsto t-\gl\one $ denotes the time-$\gl$ translation 
in the direction of the flat vector field $\one $ obtained from $1\in H$. 
(The latter represents identity elements for all the products
$\bullet_t$.)

An important consequence of Lemma \ref{lem:periods} is the
following formula due to K. Saito \cite{SaK}.
\begin{corollary}\label{int-res}
For all\/ $\al,\be\in H_{2l}(X_{0,1};\C)$, we have
\ben
(\al|\beta)=
- \bigl(I^{(0)}_\al(t,0),E\bullet_t I^{(0)}_\beta(t,0)\bigr) \,.
\een
\end{corollary}
\begin{proof}
According to \leref{lem:periods}, the right-hand side is independent of $t$. It is also monodromy invariant, therefore, up to
a constant it must coincide with the intersection pairing. It remains
only to verify that the proportionality coefficients is  $1$, which
may be reduced to the case of an $A_1$-singularity by letting $t$ approach a
generic point on the discriminant. 
\end{proof}
 
Now we will compute explicitly the period integrals $I^{(k)}_\al(0,\la)$ in
the case of simple singularities. In this case, we may choose $l=1$
and we can assign uniquely a degree $\chi_i$ to $x_i$, so that the
polynomial $f(x)$ is \emph{weighted homogeneous} of degree $1$ (see Table \ref{table1} in Sect.\ \ref{ssing}). Furthermore, we
can fix the flat coordinates in such a way that if we set $\deg\tau^i=1-d_i$ then $F(t,x)$ is still weighted
homogeneous of degree $1$. In particular, the polynomials $\d_i F(t,x)$ are
weighted homogeneous of degree $d_i$. 

Then the integral \eqref{periods}
for $k=t=0$ assumes the form
\beq\label{per-0}
(I^{(0)}_\al(0,\la),\d_i) = \la^{s_i} \langle v^i,\al\rangle \,,
\eeq
where 
\ben
s_i = d_i- \frac{d}{2}-\frac{1}{2} \,, \qquad
d=\sum_{i=0}^{2l} (1-2\chi_i) = 1-\frac{2}{h}\, ,
\een
and $v^i$ are some constant sections of the middle cohomology
bundle (see \eqref{per-map} below).

By definition, the analytic continuation in $\la$ along a
counter-clockwise loop around $\la=0$ is equivalent to the parallel
transport of $\al$ along that loop, i.e., 
\ben
e^{2\pi \sqrt{-1}\, s_i} \la^{s_i} \langle v^i,\al\rangle = \la^{s_i}
\langle v^i,\si(\al)\rangle \,.
\een
We obtain that $v^i$ are eigenvectors of the classical monodromy
$\si$, which according to \coref{int-res} satisfy the following orthogonality relations:
\ben
\si(v^i) = e^{-2\pi \sqrt{-1}\, s_i} v^i,\qquad
(v^i|v^j)=(\d_i,\d_j) \,.
\een
In particular, $s_i=-m_i/h$, where $m_i$ are the
exponents of the corresponding simple Lie algebra. 
In other words, the period mapping 
\beq\label{per-map}
\d_i \mapsto v^i:=\left. (2\pi)^{-l} \Big( \d_\la^l \int
  \d_i F\frac{\omega}{dF} \Big)\right|_{(t,\la)=(0,1)} 
\eeq
is an isomorphism between the space of flat vector fields and $\lieh^*,$ i.e.,
\ben
\Gamma(B,\T_B)^\nabla \ \cong\ H^{2l}(X_{0,1};\C) \,,
\een
which identifies the residue pairing with the intersection pairing. 
\comment{
In addition, 
the following formulas hold:
\beq\label{pk}
I^{(k)}_\al(0,\la) = \sum_{i=1}^N s_a(s_a-1)\cdots (s_a-k+1) \la^{s_a-k}
\langle A_a,\al\rangle d\tau^a,\quad k\geq 0,
\eeq
and
\beq\label{p-k-1}
I^{(-k-1)}_\al(0,\la) = \sum_{a=1}^N
\frac{\la^{s_a+k+1}}{(s_a+1)\cdots(s_a+k+1)} 
\langle A_a,\al\rangle d\tau^a,\quad k\geq 0.
\eeq}

\subsection{Stationary phase asymptotic and calibration}\label{sec3_4}

Let $u_i(t)$ ($1\leq i\leq N$) be the critical values of $F(t,\cdot)$. For a generic $t$, they form a local coordinate system on $B$ in which the Frobenius multiplication and the residue pairing are diagonal. Namely,
\ben
\d/\d u_i \, \bullet_t\, \d/\d u_j = \delta_{ij}\d/\d u_j \,,\quad 
\(\d/\d u_i,\d/\d u_j \) = \delta_{ij}/\Delta_i \,,
\een
where $\Delta_i$ is the Hessian of $F$ with respect to the volume form $\omega$ at the critical point corresponding to the critical value $u_i.$ 
Therefore, the Frobenius structure is \emph{semi-simple}.

We denote by $\Psi_t$ the following linear isomorphism
\ben
\Psi_t\colon \C^N\rightarrow T_t B \,,\qquad 
e_i\mapsto \sqrt{\Delta_i}\d/\d u_i \,,
\een
where $\{e_1,\dots,e_N\}$ is the standard basis for $\C^N$.
Let $U_t$ be the diagonal matrix with entries $u_1(t),\ldots, u_N(t)$. 

According to Givental \cite{G1}, the system of differential equations (cf.\ \eqref{frob_eq3}, \eqref{frob_eq4})
\begin{align}\label{de_1}
z\d_i J(t,z) &= \d_i\bullet_t J(t,z) \,,\qquad\quad 1\leq i\leq N \,,
\\
\label{de_2}
z\d_z J(t,z) &= (\theta-z^{-1} E\bullet_t) J(t,z)
\end{align}
has a unique formal asymptotic solution of the form $\Psi_t R_t(z) e^{U_t/z}$, 
where 
\ben
R_t(z)=1+R_1(t)z+R_2(t)z^2+\cdots \,,
\een
and $R_k(t)$ are linear operators on $\C^N$ uniquely determined from the differential 
equations \eqref{de_1} and \eqref{de_2}.  

Introduce the formal series
\beq\label{falpha}
\f_\al(t,\gl,z) = \sum_{k\in \Z} I^{(k)}_\al(t,\gl) \, (-z)^k \,, \qquad
\al\in\lieh \,.
\eeq
Then equations \eqref{periods:de1}, \eqref{periods:de2} imply:
\begin{align}\label{periods:de3}
\d_i \f_\al(t,\gl,z) &= -\d_i\bullet_t \bigl(\d_\gl \f_\al(t,\gl,z)\bigr)\,, \qquad\quad 1\leq i\leq N \,, 
\\ \label{periods:de4}
\gl \d_\gl \f_\al(t,\gl,z) &= \Bigl(-z\d_z-z^{-1} E\bullet_t+\theta-\frac12\Bigr) \f_\al(t,\gl,z) \,.
\end{align}
The following result is due to Givental \cite{G3}.
\begin{proposition}\label{vanishing_a1}
Let\/ $t\in B$ be generic and\/ $\be$ be a vanishing cycle vanishing over the point\/ $(t,u_i(t))\in \Si$. Then for all\/ $\la$ near\/ $u_i(t)$, we have
\ben
\f_{\be}(t,\gl,z) = \Psi_t R_t(z) \sum_{k\in \Z} (-z\d_\gl)^k \, \frac{2e_i}{\sqrt{2(\gl-u_i(t))}} \,.
\een
\end{proposition}

\medskip

One can think of the connection operator in \eqref{frob_eq3} as an isomonodromic family 
of connection operators $\nabla_t$ over $\C\setminus\{0\}$, parameterized by $t \in B$.
We introduce a gauge transformation $\mathcal{S}_t(z)$ of the form 
\beq\label{gauge}
\mathcal{S}_t(z)=1+S_1(t) z^{-1}+S_2(t)z^{-2}+\cdots \,,\qquad
S_k(t)\in\End H
\eeq
that satisfies the differential equations \eqref{de_1} and 
conjugates $\nabla_t$ and $\nabla_0$:
$$ 
\nabla_{t} = \mathcal{S}_{t}\ \nabla_0\ \mathcal{S}_{t}^{-1} \,, \qquad
\nabla_0 = \d_z-z^{-1}\theta + z^{-2}\rho \,,
$$
where $\rho=E\bullet_0$ is the constant part of the Euler vector field.
In general, such a gauge transformation is not uniquely determined. However, it is not hard to see 
that if we impose the initial condition $\mathcal{S}_t(z)|_{t=0} = 1$ then such a gauge 
transformation exists and is unique. 

\begin{proposition}\label{translation}
We have\/ $\mathcal{S}_t(z) \, \f_\al(0,\gl,z) = \f_\al(t,\gl,z)$ for\/ $\gl$ in a neighborhood of\/ $\infty.$
\end{proposition}
\begin{proof}
This follows from the differential equations \eqref{de_1}, \eqref{periods:de3}. 
\end{proof}

\section{Symplectic loop space formalism}\label{sec4}

The goal of this section is to introduce Givental's quantization
formalism (see \cite{G2}) and use it to define the higher genus potentials in
singularity theory. We continue using the notation of Sect.\ \ref{s2}.

\subsection{Symplectic structure and quantization}\label{ssympl}

As in Sect.\ \ref{s2}, let $H$ be the space of flat vector fields on $B$. 
The space $\H:=H(\!(z^{-1})\!)$ of formal Laurent series in $z^{-1}$ with
coefficients in $H$ is equipped with the following \emph{symplectic form}: 
\ben
\Omega(\phi_1,\phi_2):=\Res_z \(\phi_1(-z),\phi_2(z)\) \,,
\qquad \phi_1,\phi_2\in\H \,,
\een 
where, as before, $(,)$ denotes the residue pairing on $H$
and the formal residue $\Res_z$ gives the coefficient in front of $z^{-1}$.

Let $\{\d_i\}_{i=1}^{ N}$ and $\{\d^i\}_{i=1}^N$ be dual bases of $H$ with respect to the residue pairing.
Then
\ben
\Om(\d^i(-z)^{-k-1}, \d_j z^l) = \de_{ij} \de_{kl} \,.
\een
Hence, a Darboux coordinate system is provided by the linear functions $q_k^i$, $p_{k,i}$ on $\H$ given by:
\ben
q_k^i = \Om(\d^i(-z)^{-k-1}, \cdot) \,, \qquad
p_{k,i} = \Om(\cdot, \d_i z^k) \,.
\een
In other words,
\ben
\phi(z) = \sum_{k=0}^\infty \sum_{i=1}^N q_k^i(\phi) \d_i z^k +  
\sum_{k=0}^\infty \sum_{i=1}^N p_{k,i}(\phi) \d^i(-z)^{-k-1} \,,
\qquad \phi\in\H \,.
\een  
The first of the above sums will be denoted $\phi(z)_+$ and the second $\phi(z)_-$.

The \emph{quantization} of linear functions on $\H$ is given by the rules:
\ben
\widehat q_k^i = \hbar^{-1/2} q_k^i \,, \qquad
\widehat p_{k,i} = \hbar^{1/2} \frac{\d}{\d q_k^i} \,.
\een
Here and further, $\hbar$ is a formal variable. We will denote by $\C_\hbar$ the field
$\C(\!(\hbar^{1/2})\!)$.

Every $\phi(z)\in\H$ gives rise to the linear function $\Om(\phi,\cdot)$ on $\H$,
so we can define the quantization $\widehat\phi$. Explicitly,
\beq\label{phihat}
\widehat\phi = -\hbar^{1/2} \sum_{k=0}^\infty \sum_{i=1}^N q_k^i(\phi) \frac{\d}{\d q_k^i}
+ \hbar^{-1/2} \sum_{k=0}^\infty \sum_{i=1}^N p_{k,i}(\phi) q_k^i \,.
\eeq
The above formula makes sense also for $\phi(z)\in H[[z,z^{-1}]]$ if we interpret $\widehat\phi$
as a formal differential operator in the variables $q_k^i$ with coefficients in $\C_\hbar$.

\begin{lemma}\label{lphihat}
For all\/ $\phi_1,\phi_2\in\H$, we have\/ $[\widehat\phi_1, \widehat\phi_2] = \Om(\phi_1,\phi_2)$.
\end{lemma}
\begin{proof}
It is enough to check this for the basis vectors $\d^i(-z)^{-k-1}$, $\d_i z^k$, in which case it is true by definition.
\end{proof}

\subsection{Quantization of quadratic Hamiltonians}\label{squant}

It is known that both series $\mathcal{S}_t(z)$ and
\ben
\mathcal{R}_t(z):=\Psi_t R_t(z)\Psi_t^{-1}
\een
(see Sect.\ \ref{sec3_4}) are symplectic transformations. Moreover, they both have the form $e^{A(z)},$ where $A(z)$ is an infinitesimal symplectic transformation. 

A linear operator $A(z)$ on $\H:=H(\!(z^{-1})\!)$ is infinitesimal symplectic if and only if the map $\phi\in \H \mapsto A\phi\in \H$ is a Hamiltonian vector field with a Hamiltonian given by the quadratic function $h_A(\phi) = \frac{1}{2}\Omega(A\phi,\phi)$. 
By definition, the \emph{quantization} of $e^{A(z)}$ is given by the differential operator $e^{\widehat{h}_A},$ where the quadratic Hamiltonians are quantized according to the following rules:
\ben
(p_{k,i}p_{l,j})\sphat = \hbar\frac{\d^2}{\d q_k^i\d q_l^j} \,,\quad 
(p_{k,i}q_l^j)\sphat = (q_l^jp_{k,i})\sphat = q_l^j\frac{\d}{\d q_k^i} \,,\quad
(q_k^iq_l^j)\sphat = \frac1{\hbar} q_k^iq_l^j \, .
\een

\subsection{Total descendant potential}\label{giv}
Let us make the following convention. Given a vector 
\ben
\q(z) = \sum_{k=0}^\infty q_k z^k \in H[z] \,, \qquad
q_k=\sum_{i=1}^N q_k^i\d_i \in H \,,
\een
its coefficients give rise to a vector sequence  $q_0,q_1,\dots$.
By definition, a {\em formal  function} on $H[z]$,
defined in the formal neighborhood of a given point $c(z)\in H[z]$, is a formal power
series in $q_0-c_0, q_1-c_1,\dots$. Note that every operator acting on
$H[z]$ continuously in the appropriate formal sense induces an
operator acting on formal functions.

\begin{example}\label{ewktf}
The \emph{Witten--Kontsevich tau-function} is the following generating series:
\beq\label{D:pt}
\D_{\rm pt}(\hbar;Q(z))=\exp\Big( \sum_{g,n}\frac{1}{n!}\hbar^{g-1}\int_{\overline{\M}_{g,n}}\prod_{i=1}^n (Q(\psi_i)+\psi_i)\Big),
\eeq
where $Q_0,Q_1,\ldots$ are formal variables, and $\psi_i$ ($1\leq
i\leq n$) are the first Chern classes of the cotangent line bundles on
$\overline{\M}_{g,n}$ (see \cite{W1,Ko1}).
It is interpreted as a formal
function of $Q(z)=\sum_{k=0}^\infty Q_k z^k\in \C[z]$, defined in the
formal neighborhood of $-z$. In other words, $\D_{\rm pt}$ is a formal power series
in $Q_0,Q_1+1,Q_2,Q_3,\dots$ with coefficients in $\C(\!(\hbar)\!)$.
\end{example}

Let $t\in B$ be a {\em semi-simple} point, so that the critical values
$u_i(t)$ ($1\leq i\leq N$) of $F(t,\cdot)$ form a coordinate
system. Recall also the flat coordinates
$\tau=(\tau^1(t),\dots,\tau^N(t))$ of $t$. 
We now introduce the main object of our study.

\begin{definition}\label{dtdp}
The \emph{total descendant potential} of a singularity of type $X_N$ is the
following formal function on $H[z]$ defined in the formal neighborhood
of $\tau-\one z$:
\beq\label{DXN}
\D_{X_N}(\q(z))= e^{F^{(1)}(t)}\, \widehat{\S}^{-1}_t\,\widehat{\Psi}_t\,\widehat{R}_t\,e^{\widehat{U_t/z}}\,\prod_{i=1}^N \D_{\rm pt}(\hbar\,\Delta_i; \sqrt{\Delta_i}Q^i(z)) \,,
\eeq
where the factor $F^{(1)}(t)$ (called the genus-$1$ potential) is
chosen so that it makes the formula independent of $t$. 
\end{definition}

As we discussed in Sect.\ \ref{sfrobgiv}, equation \eqref{DXN} is Givental's formula, which we now take as a definition.
Let us examine more carefully the quantized action of the operators in this formula.

\subsection{The action of the asymptotical operator}\label{asop}
The operator ${\widehat{U_t/z}}$ is known to annihilate the
Witten--Kontsevich tau-function. Therefore, 
$e^{\widehat{U_t/z}}$ is redundant and it can be dropped from the
formula. By definition, $\widehat{\Psi}_t$ is the following change of variables: 
\ben
\q(z) = \Psi_t\sum_{i=1}^NQ^i(z)e_i \,,\quad \text{i.e.,}\quad \sqrt{\Delta_i}Q_k^i =\sum_{j=1}^N (\d_j u_i)\, q_k^j \,.
\een
Put $\widehat{\mathcal{R}}_t=\widehat\Psi_t \widehat R_t\widehat\Psi_t^{-1}$ and 
\ben
\leftexp{i}{\q}(z) = \sum_{k=0}^\infty \sum_{j=1}^N q_k^j(\d_j u_i) z^k \,.
\een
Then the total descendant potential assumes the form:   
\beq\label{dxna}
\D_{X_N}(\q(z))=e^{F^{(1)}(t)} \widehat{\mathcal{S}}^{-1}_t \A_t(\q(z)) \,,
\eeq
where  
\beq\label{ancestor}
\A_t(\q(z)) = \widehat{\mathcal{R}}_t\ \prod_{i=1}^N\, \D_{\rm pt}(\hbar\Delta_i;\leftexp{i}{\q}(z)) \in \C_\hbar[[q_0,q_1+\one,q_2\dots]]
\eeq
is the so-called {\em total ancestor potential} of the singularity. As before, $\C_\hbar:=\C(\!(\hbar^{1/2})\!)$.

The action of the operator $\widehat{\mathcal{R}}_t$ on formal functions, whenever it makes sense, is given as follows.
\begin{lemma}[Givental \cite{G2}]\label{lrfock}
We have
\ben
\widehat{\mathcal{R}}_t\,F(\q) = \left.\Big(e^{\frac{\hbar}{2}V\d^2}F(\q)\Big)\right|_{\q\mapsto \mathcal{R}^{-1}_t\q} \,,
\een
where\/ $V\d^2$ is the quadratic differential operator 
\ben
V\d^2 = \sum_{k,l=0}^\infty \sum_{i,j=1}^N (\d^i,V_{kl}\d^j) \, \frac{\d^2}{ \d q_k^i \d q_l^j}
\een
whose coefficients\/ $V_{kl}$ are given by 
\ben
\sum_{k,l=0}^\infty (-1)^{k+l} V_{kl}(t) z^k w^l = \frac{\leftexp{T}{R}_t(z)R_t(w)-1}{z+w} 
\een
and\/ $\leftexp{T}{R}_t(z)$ denotes the transpose of\/ $R_t(z)$.
\end{lemma}

The substitution $\q\mapsto \mathcal{R}^{-1}_t\q $ can be written more explicitly as follows:
\ben
q_0\mapsto q_0,\quad q_1\mapsto \overline{R}_1(t)q_0 + q_1,\quad q_2\mapsto \overline{R}_2(t)q_0 +\overline{R}_1(t)q_1 + q_2 \,,\dots
\een
where 
\ben
\mathcal{R}^{-1}_t(z)=1+\overline{R}_1(t)z+\overline{R}_2(t) z^2+\cdots \,.
\een
Note that this substitution is not a well-defined
operation on the space of formal functions. This complication,
however, is offset by a certain property of the
Witten--Kontsevich tau-function, which we will now explain.

By definition, an {\em asymptotical function} is a formal function  of the type:
\ben
\A(\q) = \exp \Big(\sum_{g=0}^\infty F^{(g)}(\q)\hbar^{g-1}\Big) \,.
\een
Such a function is called {\em tame} if the following $(3g-3+r)$-jet constraints are satisfied:
\ben
\frac{\d^{r} F^{(g)}}{\d q_{k_1}^{i_1}\cdots \d
  q_{k_r}^{i_r}}\Bigg|_{\q = 0} = 0 \quad \text{ if } \quad k_1+\cdots
+ k_r > 3g-3+r \,. 
\een   
The Witten--Kontsevich tau-function (up to the shift $q_1\mapsto q_1+1$) is tame for dimensional reasons, since $\dim\mathcal{M}_{g,r}=3g-3+r$.

Due to Givental \cite{G2}, the action of the operator
$\widehat{\mathcal{R}}_t$ on tame functions is well defined. Moreover, the
resulting series is also a tame asymptotical function. In particular,
the total ancestor potential $\A_t$ is a tame asymptotical function.

\subsection{The action of the calibration}\label{cali}
The quantized symplectic transformation $\widehat{\mathcal{S}}^{-1}_t$ acts on formal functions as follows.
\begin{lemma}[Givental \cite{G2}]\label{lsfock}
We have
\beq\label{S:fock}
\widehat{\mathcal{S}}^{-1}_t\, F(\q) = e^{\frac1{2\hbar} W\q^2}F\bigl((\mathcal{S}_t\q)_+\bigr) \,,
\eeq
where\/ $W\q^2$ is the quadratic form 
$$
W\q^2 = \sum_{k,l=0}^\infty (W_{kl}q_l,q_k)
$$ 
whose coefficients are defined by
\ben
\sum_{k,l=0}^\infty W_{kl}(t) z^{-k}w^{-l}=\frac{\leftexp{T}{\mathcal{S}}_t(z)\mathcal{S}_t(w)-1}{z^{-1}+w^{-1}} \,.
\een
\end{lemma}

The $+$ sign in \eqref{S:fock} means truncation of all negative powers of $z$, i.e., in $F(\q)$ we have to substitute (cf.\ \eqref{gauge}):
\ben
q_k\mapsto q_k + S_1(t) q_{k+1}+S_2(t) q_{k+2}+\cdots \,,\quad k=0,1,2,\dots\ .
\een
This operation is well defined on the space of formal power series.

\begin{lemma}\label{lstiso}
We have an isomorphism
\ben
\widehat{\S}_t^{-1} \colon \C_\hbar[[q_0,q_1+\one,q_2,\dots]]\to \C_\hbar[[q_0-\tau,q_1+{\bf 1},q_2,\dots]] \,.
\een
\end{lemma}
\begin{proof}
We only need to check that $S_1(t){\bf 1}=\tau(t)$, which
can be proved as follows. Since $\mathcal{S}_t(z)$
satisfies the differential equations $z\d_i\mathcal{S}_t = \d_i\bullet_t
\mathcal{S}_t$, by comparing the coefficients in front of $z^0$, we
get $\d_i(S_1(t){\bf 1}) = \d_i$. Hence, $S_1(t){\bf 1}=\tau$ up to additive constants. 
But the constants must be $0$ because $S_1(0)=0$, since $\mathcal{S}_t(z)|_{t=0}=1$. 
\end{proof}

\comment{
The above discussion implies that $\D_B\in \C_\hbar[[q_0-\tau,q_1+{\bf
  1},q_2,\dots]]$. On the other hand, the total descendant potential
is independent of $\tau$ in the sense that $\d_a \D_B=0$ for all
$a=1,\dots,N$. Setting $\tau=q_0$ we get that $\D_B$ is a formal
series in $q_1+1,q_2,\dots$ whose coefficients are holomorphic in
$B'$, the open subset of semi-simple points in $B$. The coefficients
may have poles along the {\em caustic} locus $B\setminus{B'}.$ It is a
conjecture of Givental that $\D_B$ extends across the caustic, so it
should be a formal power series in $q_0$. The conjecture is proved
 for $A_N$ singularities in \cite{G3}, while for the other
simple singularities it can be deduced from the work of
Fan--Jarvis--Ruan \cite{FJR2}. Nevertheless, we do not assume this
result. We drop the index $B$ from $\D_B$ and for a given semi-simple
point $t\in B$ with flat coordinates $\tau=(\tau^1,\dots,\tau^N)$ we
denote by $\D_\tau\in \C_\hbar[[q_0-\tau,q_1+{\bf 1},q_2,\dots]]$ the
Taylor expansion of $\D$ at the point $\tau$.    }

\section{Analytic continuation of the vertex algebra representation }\label{savar}

In this section, we construct an analytic continuation of the representation of the vertex algebra $\F$
on the twisted Fock space. More precisely, we will construct 
formal differential operators $X_t(a,\la)$ for $a\in \F$,
whose coefficients are multivalued analytic functions of $(t,\la) \in (B\times \C)\setminus\Sigma$
branching along the discriminant $\Sigma$. These operators possess remarkable properties
and will be crucial for the proof of our main theorem.

\subsection{Twisted Fock space}\label{stwfock}
As in Sect.\ \ref{ssing}, consider a simple singularity of type $X_N$ $(X=A,D,E)$ with a Milnor lattice $Q=H_{2l}(X_{0,1};\Z)$.
Then $Q$ is a root lattice of type $X_N$ and the set $R\subset Q$ of vanishing cycles is the corresponding root system. 
The bilinear form $(\cdot|\cdot)$, equal $(-1)^l$ times the intersection form,
is such that $|\al|^2 = (\al|\al) =2$ for $\al\in R$. The monodromy group $W$ coincides with the Weyl group, and the classical monodromy is a 
Coxeter element $\si\in W$. 

Let $\{v_j\}_{j=1}^N$ be a basis for $\lieh=\C\otimes_\Z Q = H_{2l}(X_{0,1};\C)$ consisting of eigenvectors of $\si$:
\ben
\si(v_j) = e^{2\pi\ii m_j/h} v_j \,, \qquad j=1,\dots,N
\een
(see Table \ref{table1} in Sect.\ \ref{ssing}).
Since $(\cdot|\cdot)$ is $W$-invariant, we have $(v_i|v_j)=0$ unless $m_i+m_j=h$, which is equivalent to $i+j=N+1$. Hence, the dual basis $\{v^j\}_{j=1}^N$ can be chosen
$v^j=v_{N+1-j}$.

The $\si$-twisted Heisenberg algebra $\hat\lieh_\si$ from Sect.\ \ref{twheis} has a basis $\{K, {v_j}_{(k-m_j/h)} \}_{k\in\Z, 1\leq j\leq N}$.
Its irreducible highest weight representation, the $\si$-twisted Fock space $\F_\si$ can be identified with the space of polynomials in
${v_j}_{(-k-m_j/h)}$ where $k=0,1,\dots$ and $1\leq j\leq N$. 
We will slightly modify this representation. Introduce
\beq\label{fhbar}
\F_\hbar := \C_\hbar \otimes_\C \F_\si \cong \C_\hbar[\q] \quad \text{where} \quad
\q=\{q_k^j\}_{k=0,1,2,\dots}^{j=1,\dots,N} \,.
\eeq
Then $\hat\lieh_\si$ acts on $\F_\hbar$ as follows:
\beq\label{twh1}
\begin{split}
{v_j}_{(-k-m_j/h)} &= \frac{ \hbar^{-1/2} \, q_k^j }{ (m_j/h)_k } \,,
\\
v^j_{(k+m_j/h)} &= (m_j/h)_{k+1} \, \hbar^{1/2} \, \frac{\d}{\d q_k^j} \,, \qquad
 \,.
\end{split}
\eeq
for $k=0,1,2,\dots$ and $1\leq j\leq N$, where
\ben
(x)_k := x(x+1)\cdots(x+k-1) \,, \qquad (x)_0 :=1 \,.
\een

As in Sect.\ \ref{twheis}, $\F_\hbar$ is a $\si$-twisted representation of the vertex algebra $\F$, with generating fields
given by \eqref{twheis3}:
\beq\label{twh2}
Y(v^j,\ze) = \sum_{k\in\Z} v^j_{(k+m_j/h)} \, \ze^{-k-1-m_j/h} \,.
\eeq

\subsection{Period representation}\label{sperep}
For $\al\in\lieh$, we let
\ben
X_t(\al, \la) = \d_\la\widehat\f_\al(t,\la) = \d_\la\widehat\f_\al(t,\la)_+ + \d_\la\widehat\f_\al(t,\la)_-
\een
be the quantization of $\d_\la\f_\al(t,\la,z)$ (see \eqref{falpha} and Sect.\ \ref{ssympl}). 
More explicitly, due to \eqref{phihat}, we have:
\begin{align}\label{quant1}
\d_\la\widehat\f_\al(t,\la)_+ &= \sum_{k=0}^\infty \sum_{i=1}^N \, (-1)^{k+1} \, (I_\al^{(k+1)}(t,\la), \d^i) \,
\hbar^{1/2} \, \frac{\d}{\d q_k^i} \,,
\\ \label{quant2}
\d_\la\widehat\f_\al(t,\la)_- &= \sum_{k=0}^\infty \sum_{i=1}^N \, (I_\al^{(-k)}(t,\la), \d_i) \,
\hbar^{-1/2} \, q_k^i  \,,
\end{align}
where $\{\d^i\}$ and $\{\d_i\}$ are dual bases for $H$ with respect to the residue pairing $(,)$. 
Using the period isomorphism \eqref{per-map}, we may further arrange,
by changing the basis of flat vector fields if necessary, that $\d_i=v^i$ and $\d^i=v_i$.
By \eqref{per-0}, we get that near $\la=\infty$
\ben
(I_{v^j}^{(k+1)}(0,\la), \d^i) = \de_{ij} \d_\la^{k+1} \la^{-m_j/h}
= \de_{ij} (-1)^{k+1} (m_j/h)_{k+1} \, \la^{-k-1-m_j/h}
\een
and
\ben
(I_{v_j}^{(-k)}(0,\la), \d_i) 
= \de_{ij} \d_\la^{-k} \la^{-1+m_j/h}
= \frac{ \de_{ij} \la^{k-1+m_j/h} }{ (m_j/h)_k } 
\een
for all $k\geq0$. Comparing with \eqref{twh1}, \eqref{twh2}, we obtain the following.

\begin{lemma}\label{ltwh}
For\/ $\al\in\lieh$ and\/ $\la$ close to $\infty$, 
$X_0(\al, \la)$ coincides with\/ $Y(\al, \la)$ acting on\/ $\F_\hbar$.
\end{lemma}

As a consequence, the coefficients of the Laurent expansions near $\la=\infty$ of $X_0(\al, \la)$ for $\al\in\lieh$ satisfy
the commutation relations of the $\si$-twisted Heisenberg algebra $\hat\lieh_\si$. On the other hand, by \leref{lphihat}, we have
\ben
[X_t(\al, \la), X_t(\be, \mu)] = \Om(\d_\la\f_\al(t,\la,z), \d_\mu\f_\be(t,\mu,z)) \,.
\een
Since $\f_\al(t,\la,z) = \mathcal{S}_t(z) \f_\al(0,\gl,z)$ and $\S_t(z)$ is a symplectic transformation, 
the above Lie bracket is independent of $t$. Let us denote by $Y_t^\infty(\al,\la)$ the Laurent series expansion of
$X_t(\al,\la)$ near $\la=\infty$.  

\begin{corollary}\label{ctwh}
For all\/ $t\in B$, the operator series $Y^\infty_t(\al, \la)$ generate a\/
$\si$-twisted representation of the vertex algebra\/ $\F$ on\/ $\F_\hbar$.
\end{corollary}
Our next goal is to express the operator series $Y_t^\infty(a,\la)$,
$a\in \F$ in terms of normally ordered products of the generating
fields and certain functions, which we call propagators. 


\subsection{Propagators and normally ordered product}\label{spnop}

We define the \emph{normally ordered product} of formal differential operators $D_1,\dots,D_r$ in $\q$
by putting all partial derivatives to the right of all variables, and we use the notation ${:}D_1\cdots D_r{:}$.
By definition, ${:}D_1\cdots D_r{:}$ remains the same if we
permute the factors.

In order to define the propagators, let us look at the identity
\ben
Y_t^\infty(\al,\mu)Y_t^\infty(\beta,\la) =
{:}Y_t^\infty(\al,\mu)Y_t^\infty(\beta,\la){:} +
P^{\infty}_{\al,\beta}(t,\mu,\la)\,,
\een
where $\al,\be\in\lieh$ and $P^{\infty}_{\al,\beta}(t,\mu,\la)$ is the Laurent
expansion at $\la=\infty$ and $\mu=\infty$ of the following series
\beq\label{propser}
\bigl[ \d_\mu\widehat\f_\al(t,\mu)_+ , \d_\la\widehat\f_\be(t,\la)_- \bigr]
= \sum_{k=0}^\infty \, (-1)^{k+1} \, (I_\al^{(k+1)}(t,\mu), I_\be^{(-k)}(t,\la)).
\eeq
Note that the above series is convergent in the formal
$\mu^{-1}$-adic topology.

Since $Y_t^\infty(\al,\mu)$ and $Y_t^\infty(\beta,\la)$
satisfy the commutation relations of the $\si$-twisted Heisenberg algebra (cf.\ \eqref{twheis4}), we have 
\ben
(\mu-\la)^2\,P^{\infty}_{\al,\beta}(t,\mu,\la)  \in\C_\hbar(\!(\la^{-1},\mu^{-1})\!) \,.
\een
This implies that 
\ben
P^{\infty}_{\al,\beta} (t,\mu,\la)= \iota_{\mu,\la} \Big(
(\al|\beta)(\mu-\la)^{-2} + \sum_{k=0}^\infty
P^{\infty,k}_{\al,\beta}(t,\la) (\mu-\la)^k\Big),
\een
where $\io_{\mu,\la}$ denotes the expansion for $|\mu|>|\la|$ (cf.\
\eqref{iota}) and $ P^{\infty,k}_{\al,\beta} (t,\la) $ are
some formal Laurent series in $\la^{-1}$. 
The next result is reminiscent of the well-known Wick formula (see e.g.\ \cite{K2}, Theorem 3.3).
\begin{proposition}\label{dxtal}
For\/ $a\in\F$ of the form
\ben
a = \al^1_{(-k_1-1)} \cdots \al^{r}_{(-k_{r}-1)} \vac \,, \qquad
r\geq1 \,, \; \al^i\in\lieh \,, \; k_i \geq0 \,,
\een
we have
\beq\label{Wick}
Y^\infty_t(a,\la) = \sum_{J} \, \Bigl( \prod_{ (i,j)\in J } \d_\la^{(k_j)} P^{\infty,k_i}_{\al^i,\al^j}(t,\la) \Bigr)
\; {:} \Bigl( \prod_{l\in J'} \d_\la^{(k_l)} Y^\infty_t(\al^l, \la)  \Bigr) {:} \,,
\eeq
where the sum is over all collections\/ $J$ of disjoint ordered pairs\/ $(i_1,j_1),$ $\dots,$ $(i_s,j_s)$ $\subset \{1,\dots,r\}$ such that\/
$i_1<\cdots<i_s$ and\/ $i_l<j_l$ for all $l$,
and\/ $J' = \{1,\dots,r\} \setminus \{i_1,\dots,i_s,j_1,\dots,j_s\}$.
\end{proposition}
In the above formula, $J=\emptyset$ is allowed, an empty product is
considered equal to~$1$, and here and further we use the divided-powers notation $\d_\la^{(k)} := \d_\la^k/k!$.
\begin{proof}
Let us prove the proposition only for $r=2$. The general case follows
easilly by induction on $r$. Put $\al^1:=\al$, $\al^2:=\beta$, $k_1:=k$, and $k_2:=0$. 
By \prref{pnprod}, we have for $v\in \F_\hbar$
\ben
Y_t^\infty(\al_{(-k-1)}\be, \la) v = \d_\mu^{(k+2)} \bigl( (\mu-\la)^2 \, Y_t^\infty(\al, \mu) Y_t^\infty(\be, \la) v\bigr)\big|_{\mu=\la}.
\een
Using the expansion of $P_{\al,\beta}^{\infty}$ in the powers of $\mu-\la$, we get 
\ben
 {:} \bigl( \d_\la^{(k)} Y_t^\infty(\al, \la) \bigr) Y_t^\infty(\be, \la){:} v
+ P^{\infty,k}_{\al,\be}(t,\la) v \,,
\een
as claimed.
\end{proof}

Let $F_{\al^1,\dots,\al^r}(t,\la)$ be a {\em multivalued analytic}
function of $(t,\la)\in(B\times \C)\setminus\Si$ depending on
$\al^1,\dots,\al^r \in\lieh,$ i.e., this is a function holomorphic in
a neighborhood of some reference point, say $(0,1)$, which can be
extended analytically along any path $C\subset (B\times \C)\setminus\Si$.  
\begin{definition}\label{dmonow}
We say that a multivalued analytic function
$F_{\al^1,\dots,\al^r}(t,\\ \la)$
has \emph{monodromy $W$} if its analytic continuation along a loop $C$ is
$ F_{w\al^1,\dots,w\al^r}(t,\la)$ where $w\in W$ is the monodromy
operator induced by $C$ (see Sect.\ \ref{mfib}).
\end{definition}
Note that by definition $X_t(\al,\la)$, $\al\in \lieh$, is a formal differential
operator whose coefficients are multivalued analytic functions with
monodromy $W$. We will prove in Sect.\ \ref{acpf} below the following
theorem. 
\begin{theorem}\label{tprop1}
For all\/ $\al,\be\in\lieh$ the Laurent series
$P^{\infty,k}_{\al,\beta}(t,\la)$ are convergent and give rise to multivalued analytic functions
$P^{k}_{\al,\beta}(t,\la)$ with monodromy $W$. 
\end{theorem}
We will show in Sect.\ \ref{acpf} that, in fact, the series 
\beq\label{propag}
P_{\al,\be}(t,\la;\xi) 
= (\al|\be) \xi^{-2} + P^0_{\al,\be}(t,\la) \xi^{0} + P^1_{\al,\be}(t,\la) \xi^{1} + \cdots 
\eeq
is convergent for sufficiently small $|\xi|>0$, but this will not be needed here. 
We call $P_{\al,\be}(t,\la;\xi)$ the \emph{propagator} from $\al$ to
$\be$.
\begin{example}\label{eprop}
For\/ $\al,\be\in\lieh$ such that\/ 
$\si(\al) = e^{-2\pi\ii p} \al$ $(p\in\frac1h\Z)$, we have
\beq\label{prop-0}
\begin{split}
\bigl[ \d_\mu\widehat\f_\al(0,\mu)_+ , &\d_\la\widehat\f_\be(0,\la)_- \bigr]
= (\al|\be) \d_\la \io_{\mu,\la} \frac{ \la^{p} \mu^{-p} }{\mu-\la} 
\\
&= (\al|\be) \sum_{k=0}^\infty (p+k) \la^{p+k-1} \mu^{-p-k-1} \,.
\end{split}
\eeq
This follows from \leref{ltwh} and \eqref{twheis41} or, alternatively, from \eqref{propser} and \eqref{per-0}.
For\/ $|\mu|>|\la|>0$, the above series converges to a multivalued analytic function of\/ $(\la,\mu)$
with a pole of order at most\/ $2$ at\/ $\mu=\la$.
\end{example}

Let us denote by $\io_\la$ the Laurent expansion near
$\la=\infty$. Formula \eqref{Wick} and \thref{tprop1} imply that 
\beq \label{txtla}
 Y_t^\infty(a,\la)=\io_\la X_t(a,\la) \,,\qquad a\in \F\,,
\eeq
where $X_t(a,\la)$ is a formal differential operator in\/ $\q$ 
whose coefficients are polynomial expressions of  the periods and the
propagators. Namely, assuming the same
notation and conventions as in \prref{dxtal}, we have
\beq\label{wick-1}
X_t(a,\la) = \sum_{J} \, \Bigl( \prod_{ (i,j)\in J } \d_\la^{(k_j)} P^{k_i}_{\al^i,\al^j}(t,\la) \Bigr)
\; {:} \Bigl( \prod_{l\in J'} \d_\la^{(k_l)} X_t(\al^l, \la)  \Bigr) {:} \,.
\eeq

\begin{remark}\label{rvq}
The operators\/ $X_t(a,\la)$ can be defined for all\/ $a\in V_Q$, but this will not be needed here. In particular, $X_t(a,\la)$
for\/ $a\in\lieh$ and\/ $a=e^\al$ $(\al\in R)$ provide a realization of the basic representation of the affine Kac--Moody algebra\/ $\hat\lieg$
(cf.\ \cite{FGM}). The operators\/ $X_t(e^\al,\la)$ are defined in terms of the so-called \emph{vertex operators} (cf.\ \eqref{twlat12}):
\beq\label{vertop}
\Gamma_\al(t,\la) = {:} \exp \widehat\f_\al(t,\la) {:} 
= \exp\bigl( \widehat\f_\al(t,\la)_-\bigr) \exp\bigl( \widehat\f_\al(t,\la)_+\bigr) \,.
\eeq
\end{remark}

\subsection{Behavior near a critical point}\label{sxtal}

Our next goal is to understand the behavior of $X_t(a,\la)$ near a generic point $(t,u_i(t))$ on the discriminant.
We will write $u=u_i(t)$ for short, and will fix a cycle $\be\in\lieh$ vanishing over $(t,u)$.
Denote by $\io_{\la-u}$ the operation of Laurent expansion near $\la=u$, and let
\ben
Y_t^u(a, \la) := \io_{\la-u} X_t(a, \la) \,, \qquad a\in\F\,.
\een

The following properties of the propagators will be proved in Sect.\ \ref{acpf} below.
\begin{theorem}\label{tprop2} The following statements hold{\rm:}
\begin{enumerate}
\item[(a)]
 If\/ $(\al'|\be)=(\al''|\be)=0$, then the Taylor coefficients\/
$P^k_{\al',\al''}(t,\la)$ of the propagator are analytic near\/ $\la=u$.
 
\item[(b)]
 There exists\/ $r_i(t)>0$ such that for all\/ $\al\in\lieh$ the Laurent expansion of\/ $P_{\al,\be}(t,\la;\mu-\la)$ in the domain\/
$r_i(t)>|\mu-u|>|\la-u|>0$ is equal to
\beq\label{prop-seru}
\sum_{k=0}^\infty \, (-1)^{k+1} \, ( \io_{\mu-u} I_\al^{(k+1)}(t,\mu), \io_{\la-u} I_\be^{(-k)}(t,\la)) \,.
\eeq
\end{enumerate}
\end{theorem}

Let $\F_\be\subset\F$ be the vertex subalgebra generated by $\be$. Then
\beq\label{fbeta}
\F_\be \cong \C[\be_{-1},\be_{-2},\be_{-3},\dots]
\eeq
is just the Fock space for the Heisenberg algebra
$\widehat{\C\be}$. 
Similarly, let
\beq\label{fbetaperp}
\F_\be^\perp = \{ a\in\F \,|\, \be_{(n)} a = 0 \,, \;\; n\geq0 \}
\eeq
be the Fock space for the Heisenberg algebra $\widehat{(\C\be)^\perp}$, which is a subalgebra of $\F$ commuting with $\F_\be$.
Note that we have an isomorphism
\beq\label{fbetaiso}
\F_\be^\perp\otimes\F_\be \cong \F \,, \quad a\otimes b \mapsto a_{(-1)}b \,.
\eeq
Recall that, by the Picard--Lefschetz formula (see Sect.\ \ref{mfib}), the monodromy operator associated to a small loop around $(t,u)$
is the reflection $r_\be\in W$.

\begin{theorem}\label{txtla2} The following statements hold{\rm:}
\begin{enumerate}
\item[(a)]
For\/ $a\in\F_\be^\perp$, the coefficients of\/ $X_t(a,\la)$ are holomorphic functions of\/ $(t,\la)$
in a neighborhood of\/ $(t,u)$.
\item[(b)]
The map
$
b\in \F_\be \mapsto Y_t^u(b,\la)
$
is an\/ $r_\be$-twisted representation of the vertex algebra\/
$\F_\be$ on\/ $\F_\hbar$.
\item[(c)]
For $a\in\F_\be^\perp \,, \; b\in\F_\be$, we have
\ben
Y_t^u(a_{(-1)} b, \la) = Y_t^u(a, \la) Y_t^u(b, \la) \, .
\een
\end{enumerate}
\end{theorem}
\begin{proof}
(a)
If $\al\in\lieh$ is such that $(\al|\be)=0$, the coefficients of
$X_t(\al,\la)$ are invariant with respect to the local monodromy
$r_\be$, so they must be holomorphic functions of $(t,\la)$
in a neighborhood of $(t,u)$. The statement for $a\in\F$ then follows from the
definition \eqref{wick-1} of $X_t(a,\la)$ and \thref{tprop2} (a).

(b)
Note that 
\ben
[Y_t^u(\beta,\mu),Y_t^u(\beta,\la)] =
\io_{\mu-u}\io_{\la-u}
\,
\Omega(\d_\mu\f^{A_1}(u,\mu,z), \d_\la\f^{A_1}(u,\la,z)) \,,
\een
where we used \leref{lphihat} and \prref{vanishing_a1} and we denoted
by 
\beq\label{fa1}
\f^{A_1}(u,\la,z) = \pm 2 \sum_{k\in \Z} (-z\d_\la)^k (2(\la-u))^{-1/2}
\eeq
the period series for an $A_1$-singularity. It follows that the
coefficients of $Y_t^u(\be,\la)$ satisfy the $r_\be$-twisted Heisenberg
relations. Let us denote by 
\ben
b\mapsto \widetilde{Y}_t^u(b,\la)\,,\qquad b\in \F_\be
\een 
the $r_\beta$-twisted representation generated by $Y_t^u(\be,\la).$
By the same argument as in the proof of 
\prref{dxtal}, we can express $\widetilde{Y}^u_t(b,\la)$ in terms of the generating
fields $Y_t^u(\be,\la)$ and some propagators
$P_{\be,\be}^{u,k}(t,\la)$. According to \thref{tprop2} (b), we have
\ben
P_{\be,\be}^{u,k}(t,\la)=\io_{\la-u}\, P_{\be,\be}^{k}(t,\la)\,,
\een 
which implies that
\ben
\widetilde{Y}_t^u(b,\la) = \io_{\la-u}\, X_t(b,\la) = Y_t^u(b,\la) \, .
\een

(c)
We can assume that
\begin{align*}
a &= \al^1_{(-k_1-1)} \cdots \al^{r}_{(-k_{r}-1)} \vac \,, \qquad
r\geq1 \,, \; \al^i\in(\C\be)^\perp \,, \; k_i \geq0 \,,
\\
b &= \be_{(-m_1-1)} \cdots \be_{(-m_{p}-1)} \vac \,, \qquad
p\geq1 \,, \; m_i \geq0 \,.
\end{align*}
Then, by \eqref{fbetaiso},
\ben
a_{(-1)} b = \al^1_{(-k_1-1)} \cdots \al^{r}_{(-k_r-1)} 
\be_{(-m_1-1)} \cdots \be_{(-m_{p}-1)} \vac \,.
\een
Using \eqref{wick-1},
one can express $Y_t^u(a, \la)$, $Y_t^u(b, \la)$ and $Y_t^u(a_{(-1)} b, \la)$ in terms of normally ordered products and propagators.

Then to compute $Y_t^u(a, \la) Y_t^u(b, \la)$, it is enough to compute all products of the type
$A_{J'}(\la)B_{J''}(\la)$, where
\ben
A_{J'}(\la) = {:} \Bigl( \prod_{i\in J'} \d_\la^{(k_i)} Y_t^u(\al^i, \la) \Bigr) {:} \,, \quad 
B_{J''}(\la) = {:} \Bigl( \prod_{j\in J''} \d_\la^{(m_j)} Y_t^u(\be, \la) \Bigr) {:} 
\een
for
\ben
J' \subset\{1,\dots,r\} \,, \qquad J'' \subset\{1,\dots,p\} \,.
\een
The product $A_{J'}(\la)B_{J''}(\la)$ is computed using the Wick formula (see e.g.\ \cite{K2}, Theorem 3.3)
and \thref{tprop2} (b):
\ben
A_{J'}(\la)B_{J''}(\la) = \sum_{I} \, \Bigl( \prod_{ (i,j)\in I } \d_\la^{(m_j)} \io_{\la-u} P^{k_i}_{\al^i,\be}(t,\la) \Bigr)
C_I(\la) \,.
\een
Here the sum is over all collections $I$ of disjoint ordered pairs $(i,j)$ such that
$i\in J'$, $j\in J''$, and
\ben
C_I(\la) =
\; {:} \Bigl( \prod_{l\in I'} \d_\la^{(k_l)} Y_t^u(\al^l, \la)  \prod_{n\in I''} \d_\la^{(m_n)} Y_t^u(\be, \la) \Bigr) {:} \,,
\een
where 
\ben
I'=J'\setminus\{i \,|\, (i,j) \in I\} \,, \qquad I''=J''\setminus\{j \,|\, (i,j) \in I\} \,.
\een

It is not hard to see that the combinatorics of the Wick formula and  formula \eqref{wick-1} produce exactly the
identity we claim.
\end{proof}

\begin{remark}\label{rxtla2}
It is not true that the Laurent expansions of all\/ $X_t(a,\la)$ near\/ $\la=u$ give a twisted representation of\/ $\F$. Indeed, 
for\/ $\al,\al'\in\lieh$ such that\/ $(\al|\be)=(\al'|\be)=0$,
the Laurent expansions of\/ $X_t(\al,\la)$ and\/ $X_t(\al',\la)$
have only non-negative powers of\/ $\la-u$.
Thus, they cannot satisfy the commutation relations of the Heisenberg algebra (cf.\ \eqref{twheis4}).
\end{remark}

\subsection{Action on tame vectors}\label{sactam}
So far we have considered the action of $X_t(a,\la)$ on elements of $\F_\hbar$, i.e., on polynomials in $\q$ (see \eqref{fhbar}). Now we want to consider a certain completion of $\F_\hbar$. Note that for an arbitrary formal power series $v\in\C_\hbar[[\q]]$, the series $X_t(a,\la)v$ has divergent coefficients in general. 
We claim that if $v$ is a tame asymptotical function (cf. Sect. \ref{asop}), then $X_t(a,\la)v$ is a formal power series
whose coefficients are formal Laurent series in $\hbar^{1/2}$ with
coefficients finite linear combinations of the coefficients of
$X_t(a,\la).$ 


Using the natural multi-index notations, we can write
\beq\label{vecfock}
v(\hbar,\q)=\sum_{g,I} v^{(g)}_I \, \hbar^{g-1}\, \q^I \,, \qquad
I=\{i_k^l\}_{k=0,1,2,\dots}^{l=1,\dots,N} \,.
\eeq
By definition, $v$ is \emph{tame} if $3g-3+\ell(I)<\ell_z(I)$ implies that
$v^{(g)}_I=0$, where
\ben
\ell(I):=\sum_{k=0}^\infty \sum_{l=1}^N i_k^l \,,\qquad 
\ell_z(I):=\sum_{k=0}^\infty \sum_{l=1}^N k\,i_k^l \,.
\een
If we write
\beq\label{opser}
X_t(a,\la) = \sum_{I,J} \hbar^{(\ell(J)-\ell(I))/2} a_{I,J} (t,\la) \q^I \d_\q^J \,,
\eeq
then $X_t(a,\la)v$ is a formal series of the type \eqref{vecfock}
whose coefficients $\widetilde{v}^{(g)}_I$ are given by
\ben
\sum_{I',I''\,:\,I'+I'' = I} \Bigl(\sum_J
C^{I',I''}_J a_{I',J}(t,\la) v^{(g'')}_{I''+J}\Bigr) \,,
\een
where 
\ben
g''+ \frac{1}{2}(\ell(J)-\ell(I')) = g
\een
and the precise values of the combinatorial coefficients $C^{I',I''}_J\in\Z$ are irrelevant.
The first sum is always finite for a fixed $I$, while in the second one
the non-zero terms are parameterized by $J$ such that
\ben
3g''-3+\ell(I'')+\ell(J)\geq \ell_z(I'')+\ell_z(J)\,,
\een 
i.e.,
\ben
\frac{1}{2} \ell(J) + \ell_z(J) \leq 3g-3 +\frac{3}{2}\ell(I') + \ell(I'')-\ell_z(I'')\,.
\een
For fixed $g$ and $I$, there are only finitely many $J$ satisfying the
above inequality, which proves our claim.

Finally, let us point out that the formal composition of two
operators $X_t(a,\mu)$ and $X_t(b,\la)$ is a formal differential
operator whose coefficient in front of $\hbar^{(\ell(J)-\ell(I))/2}\q^I\d_\q^J$ is 
\beq\label{comp}
\sum_{I'+I'' = I} \sum_{J'+J''=J} \Bigl(\sum_K
C^{I',I'',J',J''}_K a_{I',J'+K}(t,\mu) b_{I''+K,J''}(t,\la)\Bigr)
\eeq
for some $C^{I',I'',J',J''}_K\in\Z$. 
The first two sums are always finite for
fixed $I$ and $J$. However, the sum over $K$ is infinite and so the
product can be defined only if the series is convergent in an
appropriate sense. As we saw in Sect.\ \ref{spnop} and \ref{sxtal}, this can be done if we use the Laurent series expansions at $\infty$ or $u=u_i(t)$.

Indeed, near $\infty$, the sum over $K$ in \eqref{comp} is convergent in the $\mu^{-1}$-adic topology, since
$a_{I,J}(t,\mu) \to 0$ uniformly as $J\to\infty$.
This is because $I^{(k)}_\al(t,\mu) \to 0$ as $k\to+\infty$.

Similarly, we will say that $b\in \F$ is {\em singular at} $\la=u$ if
$b_{I,J}(t,\la)\to 0$ uniformly in the $(\la-u)$-adic topology as $I\to\infty$.
Then \eqref{comp} is convergent in the $(\la-u)$-adic topology. By \thref{txtla2} (a) and \prref{vanishing_a1},  $\al\in \lieh\subset\F$ is singular at $\la=u$ if and only if $\al\in\C\beta$, where $\be$ is a cycle vanishing over $(t,u)$. 
It follows from \eqref{wick-1} and \thref{tprop2} (b) that all $b\in\F_\be$ are singular at $\la=u$.

\subsection{Intertwining operators}

Recall that the calibration operator $\mathcal{S}_t$ of the
singularity gives an isomorphism between two different completions of
the Fock space $\F_\hbar=\C_\hbar[\q]$ (see Sect. \ref{cali}):
\ben
\widehat{\S}_t^{-1} \colon \C_\hbar[[q_0,q_1+\one,q_2,\dots]]\to \C_\hbar[[q_0-\tau,q_1+{\bf 1},q_2,\dots]] \,.
\een
Using the operator series 
\ben
Y(a,\la)=\io_\la X_0(a,\la)=Y_0^\infty(a,\la)\,,\qquad a\in \F\,,
\een 
we put the structure of a $\si$-twisted $\F$-module on the completion
$\C_\hbar[[q_0-\tau,q_1+\one,q_2\dots]]$. The other completion is
equipped with the structure of a $\si$-twisted $\F$-module via the
operator series $Y_t^\infty(a,\la)$, $a\in \F$. 
\begin{lemma}\label{S}
The map\/ $\widehat{\S}_t$ is a homomorphism of\/ $\si$-twisted\/ $\F$-modules,
i.e., 
\ben
Y^\infty_t(a,\la) = \widehat{\S}_t\,Y(a,\la)\, \widehat{\S}^{-1}_t \,,\qquad 
a\in \F\,.
\een
\end{lemma}
\begin{proof}
Using \leref{lsfock} and \prref{translation}, we see that the above
equation holds for all $a=\al\in
\lieh\subset\F$. Therefore, due to the construction via the Wick
formula (see \prref{dxtal}), it is enough to compare the
propagators. 
The propagator of $\widehat{\S}_t\,Y(\al,\mu)\,
\widehat{\S}^{-1}_t$ and $\widehat{\S}_t\,Y(\be,\la)\,
\widehat{\S}^{-1}_t$ is
\ben
\d_\la\d_\mu\,
\Omega\bigl((\S_t(z)\f_\al(0,\mu,z))_+, (\S_t(z)\f_\beta(0,\la,z))_-\bigr)\,.
\een
By \prref{translation}, this is precisely $P^\infty_{\al,\be}(t,\mu,\la)$.
\end{proof}

Now let $\F^{\rm tame}_\hbar$ be the space of \emph{tame} series in $\C_\hbar[[q_0,q_1+\one,q_2,\dots]]$, as defined in Sect.\ \ref{asop} and \ref{sactam}.
Then the asymptotical operator $\widehat{\Psi}_t\widehat{R}_t$ gives an injection (see Sect.\ \ref{asop})
\beq\label{asop-2}
\widehat{\Psi}_t\widehat{R}_t\colon \C_\hbar[[Q_0, Q_1+(1,\dots,1),Q_2,\dots ]]^{\rm tame} \rightarrow
\F^{\rm tame}_\hbar \,.
\eeq
Let us assume that $t\in B$ is generic, $\la$ is close to one
of the critical values $u:=u_i(t)$, and $\be$ is a cycle
vanishing over $(t,u)$. By \thref{txtla2} (b), the space
$\F_\hbar^{\rm tame}$ of tame vectors is an $r_\be$-twisted $\F_\be$-module. 

On the other hand, applying our construction from Sect. \ref{sperep} in the case of an $A_1$-singularity, we get that the operator series 
\ben
Y^{A_1}_u(\be,\la) = \f^{A_1}(u,\la,z)\sphat
\een
(see \eqref{fa1})
induces the structure of an $r_\be$-twisted $\F_\be$-module on the space
$\C_\hbar[[Q_0^i,Q_1^i+1,Q_2^i,\dots]]^{\rm tame}$.

\begin{lemma}\label{R}
The operator\/ \eqref{asop-2} is a homomorphism of\/ $r_\be$-twisted
$\F_\be$-modules, i.e., 
\ben
\widehat{\Psi}_t\widehat{R}_t \, Y^{A_1}_u(b,\la) = Y^u_t(b,\la) \, \widehat{\Psi}_t\widehat{R}_t
 \,,\qquad b\in \F_\be\,.
\een
\end{lemma}
\begin{proof}
Due to \leref{lrfock} and \prref{translation}, the above identity holds 
for $b=\beta$, which is the generator of $\F_\be$. According to the Wick
formula (see \prref{dxtal}), it is enough to compare the propagators
of the two modules. 
The propagator of $\widehat{\Psi}_t\widehat{R}_t Y^{A_1}_u(\be,\mu)
(\widehat{\Psi}_t\widehat{R}_t)^{-1}$ and $\widehat{\Psi}_t\widehat{R}_t Y^{A_1}_u(\be,\la)
(\widehat{\Psi}_t\widehat{R}_t)^{-1}$ is 
\ben
\d_\la\d_\mu\, \Omega\bigl((\Psi_t R_t(z) \f^{A_1}(u,\mu,z))_+,(\Psi_t R_t(z)\f^{A_1}(u,\la,z))_-\bigr) \,.
\een
But
\ben
\Psi_t R_t(z)\,\f^{A_1}(u,\la,z) = \f_\be(t,\la,z)
\een
thanks to \prref{vanishing_a1}. 
\end{proof}

\section{Analytic continuation of the propagators}\label{acpf}

The goal of this section is to prove \thref{tprop1} and
\thref{tprop2}. The idea is to express the phase factors via certain  
integrals. The latter were already used in \cite{FGM}, but we need to
establish some further properties, which will alow us to extend
analytically the phase factors and the
corresponding propagators. 

\subsection{Integral representation of the phase factors}\label{irep}
Let us denote by 
\ben
\Gamma^\infty_\al(t,\la) = \io_\la \, \Gamma_\al(t,\la)\,,\qquad \al\in Q\,,
\een
the Laurent series expansion at $\la=\infty$ of the vertex operators \eqref{vertop}. 
Due to \leref{lphihat}, the product of two vertex operators is given by
\beq\label{vop-ope}
\Gamma^\infty_\al(t,\mu) \Gamma^\infty_\beta(t,\la) = B^\infty_{\al,\beta}(t,\mu,\la) \,
{:}\Gamma^\infty_\al(t,\mu) \Gamma^\infty_\beta(t,\la){:}\, ,
\eeq
where 
\ben
B^\infty_{\al,\beta}(t,\mu,\la)  = \io_\mu\io_\la\,\exp \Omega(\f_\al(t,\mu,z)_+, \f_\be(t,\la,z)_-)
\een
is the so-called {\em phase factor}.
Then, by definition (see \eqref{propser}),
\ben
P^{\infty}_{\al,\beta}(t,\mu,\la) = \d_\mu\d_\la \log B^\infty_{\al,\beta}(t,\mu,\la) \,.
\een
The goal in this section will be to prove \thref{tprop1}. In fact, we
will prove a slightly stronger statement, namely  that
the phase factors are multivalued analytic functions with monodromy
$W$. 

We will make use of line integrals in $B$, whose integrands are
1-forms defined in terms of the period vectors
$I^{(k)}_\al(t,\la)$. It is convenient to embed $B\subset B\times \C$,
$t\mapsto (t,0)$ and restrict the Milnor fibration and the
corresponding middle homology bundle to $B$. The
restriction of the discriminant $\Si\subset B\times \C$ to $B$ will be called again the
discriminant and its complement in $B$ will be denoted by
$B'$. In particular, the period vectors 
\beq\label{tinv1}
I^{(k)}_\al(t,\la)= I^{(k)}_\al(t-\la\one,0)
\eeq
may be singular only at points $(t,\la)$ such that $t-\la\one$
belongs to the discriminant.

Using the differential equations from  Lemma \ref{lem:periods}, we get
\beq\label{phase-form}
d^B\Omega(\f_\al(t,\mu,z)_+,\f_\beta(t,\la,z)_-) =
I^{(0)}_\al(t,\mu)\bullet_t I^{(0)}_\beta(t,\la) \,,
\eeq
where $d^B$ denotes the de Rham differential on $B$.
Motivated by this identity, let us consider the following family of
improper integrals depending on parameters $t,\la,$ and $\xi$:
\beq\label{integral-11}
B_{\al,\be}(t,\la;\xi) = \lim_{\ep\to 0} 
\exp\Big( \int_{-\ep\one}^{t-\la\one} I^{(0)}_\al(t',\xi)\bullet_{t'} I^{(0)}_\beta(t',0)\Big)\,,
\eeq
where the integration is along a path $C\colon [0,1]\to B$, such that the
strip 
\ben
C_{\xi}\colon [0,1]\times[0,1] \to B, \quad 
C_{\xi}(s_1,s_2) =
C(s_1)-s_2\xi \one
\een 
does not intersect the discriminant. A path $C$ with this property
will be called a \emph{$\xi$-path}. The integrand is a multivalued 
$1$-form. In order to specify its values along the strip $C_\xi$, it is
enough to assume that the integration path passes through a reference 
point, say $-\one \in B$, where the branches of the periods are
fixed in advance. 
Note that for given $(t,\la)\notin\Si$, the
integral in \eqref{integral-11} is well defined for all sufficiently
small $\xi.$ 

Using the translation invariance of the periods (see \eqref{tinv1}) and
the fact that $\f_\beta(t,\la,z)_-$ vanishes at $t=\la\one$, we get
\beq\label{phase-inf}
B^\infty_{\al,\beta}(t,\mu,\la) = \io_\la\io_{\mu,\la}\, B_{\al,\be}(t,\la;\mu-\la)\,,
\eeq
where $\io_{\mu,\la}$ is the Laurent series expansion in the region
$|\mu|>|\la|\gg 0$. In particular, the limit in \eqref{integral-11}
exists. Indeed, let us split the integration path  in
\eqref{integral-11} into two parts (see parts $I$ and $II$ in
Fig. \ref{fig:ip} below):  
\ben
\int_{-\ep\one}^{-\la\one} +\int_{-\la\one}^{t-\la\one} \,.
\een
The second integral depends holomorphically on $\la$ and $\xi$,
because the integration path and the corresponding $\xi$-strip do not
intersect the discriminant, which
means that the integrand is analytic. 

The first integral is by \eqref{phase-form} the logarithm of the phase factor for the product
of $\Gamma^\infty_\al(0,\mu)$ and $\Gamma^\infty_\be(0,\la)$, where $\mu=\la+\xi$. 
Recall from \reref{rvq} that these vertex operators provide the principal
realization of the affine Lie algebra (see \cite{FGM}). It is an easy exercise to
compute these phase factors explicitly (see e.g.\ \cite{BK1, FGM} and Sect.\ \ref{twlat}). The answer is 
\beq\label{phase:prim}
B^\infty_{\al,\beta}(0,\mu,\la)= \prod_{k=0}^{h-1} \Bigl( \mu^{1/h} - e^{2\pi\sqrt{-1}k/h} \la^{1/h} \Bigr)^{(\si^k\al|\be)}\,.
\eeq
As in the proof of \leref{ltwlat}, one can see that for $\mu=\la+\xi$ the above function has the form 
\beq\label{sing-B}
\xi^{(\al|\beta)}(1  + \cdots)\,,
\eeq
where the dots stand for some function that depends analytically on
$\xi$.

Now to prove \thref{tprop1}, it is enough to set
\ben
P_{\al,\be}(t,\la;\mu-\la) = \d_\la\d_\mu \log\, B_{\al,\be}(t,\la;\mu-\la) \,;
\een
then the integral \eqref{integral-11} provides an analytic
continuation in $(t,\la)$ along any path avoiding the discriminant,
while formula \eqref{sing-B} implies that the propagator has the
required expansion \eqref{propag}. It remains only to prove that the phase
factors have monodromy $W$, i.e., if $C\subset B$ is a closed loop
based at $t-\la\one$ (avoiding the discriminant) and $w\in W$ is the
corresponding monodromy transformation (on vanishing homology) then
the analytic continuation of $B_{\al,\be}(t,\la;\xi)$ along $C$ is the
same as $B_{w\al,w\be}(t,\la;\xi).$  

\begin{lemma}\label{mon-inv}
Assume that\/ $B_{\al,\be}(t,\la;\xi)$ is invariant under the analytic
continuation along any loop\/ $C$ such that the corresponding monodromy
transformation\/ $w$ leaves both\/ $\al$ and\/ $\be$ invariant. Then\/ $B_{\al,\be}(t,\la;\xi)$ has monodromy\/ $W$.
\end{lemma}
\begin{proof}
Let $w\in W$ be any monodromy transformation. 
We may assume that the path in the definition of the phase factor \eqref{integral-11}
passes through a point $t_0-\la_0\one$ such that $\la_0$ is sufficiently
large. Then 
\ben
B_{\al,\be}(t,\la;\xi) = A_{C_2}(B_{\al,\be}(t_0,\la_0;\xi))\,,
\een
where $C_2$ is the portion of the path from $t_0-\la_0\one$ to $t-\la\one$ and
$A_{C_2}$ denotes analytic continuation along $C_2$.  
If $C_0$ is a loop based at $t_0-\la_0\one$ such that the
corresponding monodromy transformation is $w$, then 
\beq\label{winf}
A_{C_0}(B_{\al,\be}(t_0,\la_0;\xi)) = B_{w\al,w\be}(t_0,\la_0;\xi)\,,
\eeq
because both sides are given by an integral whose Laurent series
expansion $\io_{\la_0}\io_{\mu_0,\la_0}$ with $\mu_0=\xi+\la_0$ is
$B_{w\al,w\be}^\infty(t_0,\mu_0,\la_0)$. 

Now let $C_1$ be a loop based at $t-\la\one$ whose monodromy
transformation is $w$; then 
\ben
\frac{A_{C_1}(B_{\al,\be}(t,\la;\xi))}{B_{w\al,w\be}(t,\la;\xi)} = 
\frac{A_{C_1\circ C_2}(B_{\al,\be}(t_0,\la_0;\xi))}{A_{C_2}(B_{w\al,w\be}(t_0,\la_0;\xi))} \,.
\een  
Using \eqref{winf}, we get that the above ratio is precisely
\ben
\exp\Big(\oint_C \,
I^{(0)}_\al(t',\xi)\bullet_{t'}I^{(0)}_\be(t',0)\Big)=\frac{A_C(B_{\al,\be}(t_0,\la_0;\xi))}{B_{\al,\be}(t_0,\la_0;\xi)}
\,,
\een
where 
\ben
C=C_0^{-1}\circ C_2^{-1}\circ C_1\circ C_2\,.
\een
Both $\al$ and $\be$ are fixed by the monodromy transformation along
$C$; hence the above ratio is equal to $1$.
\end{proof}

Therefore, we need to prove that the phase factors satisfy the
condition in \leref{mon-inv}. The proof follows essentially the ideas of Givental 
\cite{G3} and consists of two steps, which are formulated in the next two lemmas. 
\begin{lemma}\label{vanishing-period}
If\/ $C$ is a small\/ $\xi$-loop that goes twice around a generic point on the
discriminant, then 
\ben
 \oint_C I^{(0)}_\al(t',\xi)\bullet_{t'} I^{(0)}_\beta(t',0)
\een
is an integer multiple of\/ $2\pi \sqrt{-1}.$
\end{lemma}
\proof
For homotopy reasons, we may assume that $C$ lies on the complex line
through the generic point $t_0$ on the discriminant parallel to
$\C\one$. Then the integral can be written as
\beq\label{period}
\oint_C( I^{(0)}_\al(t_0,\xi-u), I^{(0)}_\beta(t_0,-u))du\,,
\eeq
where $u=u_i(t)$ is the critical value that gives locally near $t_0$
the equation of the discriminant: $u(t)=0$. 

If $\al$ is invariant with respect to the local monodromy around
$t_0$, then $I^{(0)}_{\al}(t',\xi)$ is analytic for all $t'$ sufficiently
close to $t_0$ and the integral vanishes identically. 
The same also applies to $\be$.
Decomposing
$\al$ and $\beta$ into invariant and anti-invariant cycles with
respect to the local monodromy, we get
\ben
\al = \al' + (\al|\gamma)\gamma/2\,,\qquad \beta=\beta'+(\beta|\gamma)\gamma/2\,,
\een
where $\gamma$ is the cycle vanishing over $t_0$. Since only the
anti-invariant parts contribute to the integral, we may assume that
$\al=\beta=\gamma$ are vanishing over $t_0$ and will have to prove
that the integral is an integer multiple of $8\pi\sqrt{-1}$. 

In this case, the period $I^{(0)}_\beta(t_0,\xi-u)$ has the
following expansion:
\ben
I^{(0)}_\beta(t_0,\xi-u) = \Psi_{t_0} 
\sum_{k=0}^\infty R_k(t_0) \d_\xi^{-k} I^{(0)}_{A_1}(u,\xi)e_i
\een
(see Sect. \ref{sec3_4}), where 
\ben
 I^{(0)}_{A_1}(u,\xi) = \frac{\pm 2}{\sqrt{2(\xi-u)}}
\een
is the period of $A_1$-singularity. 
Substituting this expansion in formula \eqref{period} and using that
$\Psi_{t_0}$ is an isometry, we obtain
\ben
\sum_{k,l=0}^\infty \oint_C\bigl(R_k(t_0)(-\d_{u})^{-k} I^{(0)}_{A_1}(u,\xi), R_l(t_0)(-\d_{u})^{-l}I^{(0)}_{A_1}(u,0)\bigr)du\,.
\een
Since $R_{t_0}(z)$ is a symplectic transformation, we have 
$$
\sum_{k+l=n}
(-1)^k 
\ \leftexp{T}{R}_l(t_0) R_k(t_0)= \delta_{n,0}\,.
$$ 
Using integration by parts, we find that only the terms with $k=l=0$
contribute to the integral, i.e., we get
\ben
\oint_C (I^{(0)}_{A_1}(u,\xi),I^{(0)}_{A_1}(u,0))du\,.
\een

The argument is in fact slightly more subtle, because $R_t(z)$ is in general a
divergent power series. The integral operator $\sum_{k\geq n}
R_k(t_0)(-\d_{u})^{-k}$, however, when applied to the period
$I^{(0)}_{A_1}(u,\xi)$ produces a convergent series and  increases the order of
the zero at $u=\xi$ by $n$. Therefore, by induction on $n$ we see
that modulo the term with $k=l=0$ the remaining part of the integrand has an
infinite order of vanishing at $u=\xi$, so it must be~$0$.

It remains only to compute the integral
\ben
\oint_C \frac{2}{\sqrt{(\xi-u)(-u)}}\, du\,,
\een
where $C$ is a closed contour going twice around $u=0$ and
$u=\xi$. The integral is easilly seen to be $\pm 8\pi
\sqrt{-1}$, which completes the proof.
\qed

\begin{lemma}\label{path-independence}
Assume that\/ $C$ is a\/ $\xi$-loop in\/ $B'$ such that the cycles\/
$\al$ and\/ $\beta$ are invariant under the parallel transport along\/
$C$. Then 
\ben
\oint_C I^{(0)}_\al(t',\xi)\bullet_{t'}
I^{(0)}_\beta(t',0) 
\een
is an integer multiple of\/ $2\pi \sqrt{-1}.$
\end{lemma}
\proof
Since the monodromy group is a finite reflection group, any monodromy
transformation $w$ that fixes $\al$ and $\beta$ can be written as a
composition of reflections with respect to hyperplanes containing both
$\al$ and $\beta$ (see \cite{Bour}, Chapter V, Section 3.3, Proposition
2). On the other hand,
the monodromy group is the quotient of the Artin--Brieskorn braid group $\pi_1(B')$ by the normal
subgroup generated by the loops going twice around generic points on the
discriminant (see \cite{AGV,Eb}).

It follows that our path $C$ is homotopic to the composition of
several paths $C_i'$ along which $\alpha$ and $\beta$ are invariant, and several
paths $C_j''$ that are simple loops going twice around generic points
on the discriminant. Clearly, we may choose $C_i'$ and $C_j''$ to be
$\xi$-loops. 

The integral over $C$ can be written as a sum of
integrals over the loops $C_i'$ and $C_j''$. Since both periods are
invariant along $C_i'$, they must be holomorphic in a disk containing
$C_i'$, which implies that the integrals along $C_i'$
vanish. By Lemma \ref{vanishing-period}, the
integrals along $C_j''$ contribute only integer multiples of
$2\pi\sqrt{-1}$, which completes the proof.
\qed

\subsection{The phase factors near a critical value}

Let us assume now that $t_0\in B$ is a generic point on the
discriminant and that the phase factor $B_{\al,\be}(t,\la;\mu-\la)$ is analytically
extended along some path for $(t,\la,\mu)$ such that $t-\la\one$ is
close to $t_0$ and $\beta$ coincides with the cycle vanishing over
$t_0$. The crirical values $(u_1(t),\dots,u_N(t))$ of $F(t,\cdot)$
form a coordinate system for $t$ near $t_0$ and the local equation of
the discriminant has the form $\{u_i(t)=0\}$ for some $i$.   

Introduce the notation
\ben
\Gamma^{u_i}_\al(t,\la) =\io_{\la-u_i}\Gamma_\al(t,\la)\,,\qquad \al\in Q\,.
\een
The vertex operator product \eqref{vop-ope} is well defined in the
$(\la-u_i)$-adic topology, and we have 
\ben
\Gamma^{u_i}_\al(t,\mu) \Gamma^{u_i}_\be(t,\la)
=B_{\al,\be}^{u_i}(t,\mu,\la) \, {:}\Gamma^{u_i}_\al(t,\mu)
\Gamma^{u_i}_\be(t,\la) {:} \,,
\een
where the phase factor can be identified (for the same reason we used
to derive \eqref{phase-inf}) with the Laurent series
expansion
\ben
\io_{\la-u_i}\io_{\mu-u_i,\la-u_i}\, \widetilde{B}_{\al,\be}(t,\la;\mu-\la)\,.
\een
Here 
\beq\label{integral-2}
\widetilde{B}_{\al,\beta}(t,\la;\xi)= \lim_{\ep\to 0}\, \exp\Big(\int_{t_0-\epsilon\one}^{t-\la\one}
I^{(0)}_\al(t',\xi)\bullet_{t'} I^{(0)}_\beta(t',0)\Big)\,, 
\eeq
where the integration is along a $\xi$-path $C$ such that $\beta$
vanishes as $t'\to t_0$, and the limit is along a straight segment such that the
line segment $[\la+\ep,\mu+\ep]$ does not intersect $0$ as $\ep$ moves
toward $0.$ Such a path exists provided $\xi:=\mu-\la$ is
sufficiently small.

\begin{theorem}\label{t-ope}
For every\/ $\al\in Q$ and every cycle\/
 $\beta$ vanishing  over\/ $t_0$, we have 
\ben
B_{\al,\beta}(t,\la;\xi) = c_{\al,\beta}\, \widetilde{B}_{\al,\beta}(t,\la;\xi)\,,
\een
where\/ $c_{\al,\beta}$ is a constant independent of\/ $t,\la$ and $\mu$.
\end{theorem}
Note that \thref{t-ope} implies \thref{tprop2} (b), because after taking $\log$
of both sides of the above identity and differentiating
$\d_\la\d_\mu$, the left-hand side becomes the analytic continuation of the
propagator $P_{\al,\be}(t,\la;\mu-\la)$ while the Laurent series
expansion of the right-hand side is \eqref{prop-seru}. The proof of \thref{tprop2}
(a) was already obtained in the previous subsection. Indeed, if $\al'$ and
$\al''$ are cycles invariant with respect to the local monodromy, then
the periods $I^{(0)}_{\al'}(t',\xi)$ and $I^{(0)}_{\al''}(t',0)$ are
analytic for all $t'$ in a neighborhood of $t_0$, which implies that
the integral \eqref{integral-11} is analytic in $(t,\la)$.

\comment{\begin{remark}
The constants $c_{\al,\beta}$ depend on the choice of a branch of the
vertex operators $\Gamma^\al(t,\la)$ and $\Gamma^\beta_t(t,\la)$ for $t-\la\one$ near
$t_0$. While the branch of $\Gamma^\beta$ up to a $\Z_2$-monodromy
is uniquely fixed, the branch of the other vertex operator is not. The
constant $c_{\al,\beta}$ is global only if $\al=\beta$. Otherwise, the above
identity is true only locally. In other words,
$\widetilde{B}_{\al,\beta}/B_{\al,\beta}$ is a local system on the
open subset of generic points on the discriminant.
\end{remark}}

Let us denote the ratio
$B_{\al,\be}(t,\la;\xi)/\widetilde{B}_{\al,\be}(t,\la;\xi)$ by 
\beq\label{ftxi}
f_{\al,\beta}(t_0,\xi) := \lim_{(\ep',\ep'')\to 0}\ \exp\Big(  
\int_{-\ep'\one}^{t_0-\ep''\one} I^{(0)}_\al(t',\xi)\bullet_{t'} I^{(0)}_\beta(t',0) \Big)\,.
\eeq
We will prove that $f_{\al,\be}(t_0,\xi)$ is analytic in a
neighborhood of $\xi=0$ and is locally constant with respect
to $t_0$. Finally, we will check that the function is homogeneous of
degree 0, so it must be a constant independent of $\xi$.

\subsection{Analyticity at $\xi=0$}
Let us fix the following notation. The space of miniversal
deformations will be presented as $B= B_{N-1}\times \C$, where $\C$ is the
coordinate line through $\one$, the so called {\em primitive direction},  and $B_{N-1}$ is the
$(N-1)$-dimensional space spanned by the remaining coordinate
axes. Given $t\in B$, we put $'t\in B_{N-1}$ for the projection of $t$
on the first factor. Let us point out that $u_i(t) =
u_i('t)+t_N,$ therefore the points on the discriminant are precisely
the points of the form $t_0={'t_0} - u_i('t_0)\one$ for some critical
value $u_i$.  

Assume now that $'t_0\in B_{N-1}$ is generic, so that the
corresponding critical values are pairwise distinct. Note that 
the radius of convergence of the Laurent series
expansion near $\la=u_i('t_0)$ of the period $I^{(0)}_\al('t_0, \la)$,
viewed as a function of $\la$ only, is
\ben
\rho_i:=\rho_i('t_0) := {\rm min}_j |u_j('t_0)-u_i('t_0)|.
\een
Let us put $x_i:=\rho_i/2-u_i('t_0)$ and fix
$'t_0+x_i\one$ as a local reference point in the corresponding disk of
convergence $D(u_i('t_0),\rho_i).$ Here we have used the following notation:
\ben
D(u,r) = \{ 't_0+(x-u)\one\ |\ |x|<r\}.
\een
Given a positive number $\rho=\rho('t_0)$ we construct the following sets
of disks (see Fig. \ref{fig:ip}):
\begin{figure}[htbp]
\begin{center}
\scalebox{0.60}{\includegraphics{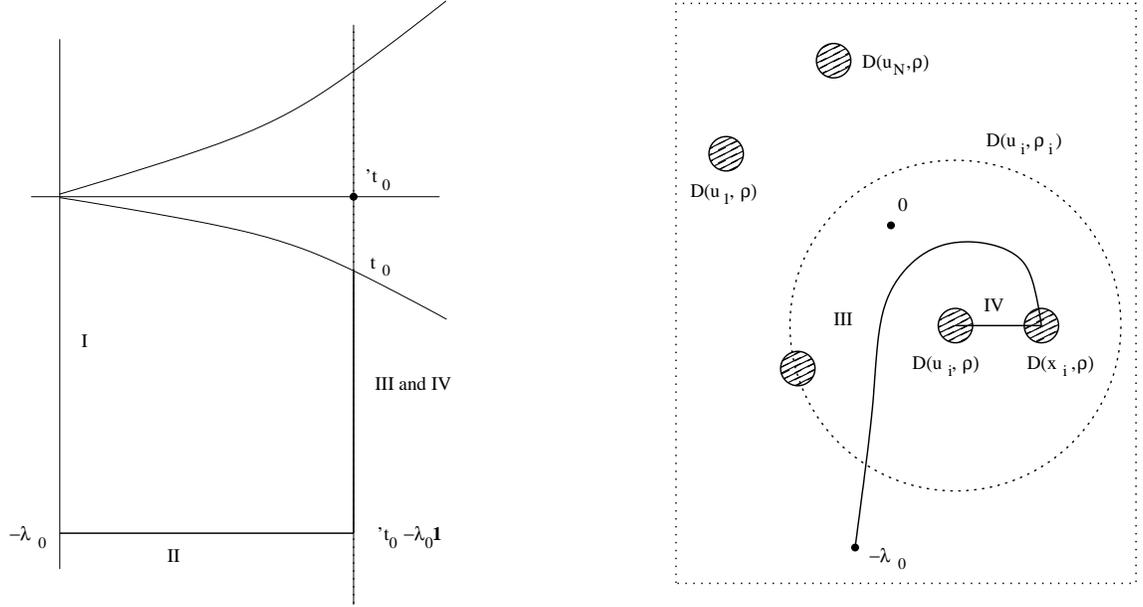}}
\caption{Integration path}
\label{fig:ip}
\end{center}
\end{figure}
\beq\label{disk-critical}
D(u_i('t_0),\rho),\quad 1\leq i\leq N
\eeq
and 
\beq\label{disk-reference}
D(x_i,\rho),\quad 1\leq i\leq N.
\eeq
We pick $\rho>0$ such that the disks \eqref{disk-critical} and
\eqref{disk-reference} are pairwise disjoint and such that
$D(x_i,\rho)$ is contained in the domain of convergence
$D(u_i('t_0),\rho_i).$ For example, if we set 
\ben
\rho('t_0)={\rm min}_{1\leq i\leq N}\ \rho_i/4 \,,
\een 
then all these requirements are satisfied. 

\begin{lemma}\label{anal}
Let\/ $t_0={'t}_0-u_i('t_0)\one$ be a generic point on the discriminant.
Then\/ $f_{\al,\beta}(t_0,\xi)$ extends analytically inside the disk\/
$|\xi|<\rho.$
\end{lemma}
\proof
Let us split the integration path in \eqref{ftxi} into four pieces (see Fig. \ref{fig:ip}): 
\beq\label{4-integrals}
\int_{-\ep'\one}^{-\la_0\one} +\int_{-\la_0\one}^{{t}_0-\la_0\one} +
\int_{t_0-\la_0\one}^{t_0+(\rho_i/2)\one} + 
\int_{t_0+(\rho_i/2)\one}^{t_0-\ep''\one}\quad ,
\eeq
where $\la_0$ can be chosen as large as we wish. The first two
integrals were already analyzed at the end of 
Sect.\ \ref{irep}. Namely, their contribution after passing to the
limit $\ep'\to 0$ is 
\beq\label{sing-0}
\log\, B_{\al,\be}(t_0,\la_0;\xi) = (\al|\be)\log\xi +\cdots\,,
\eeq
where the dots stand for some function that depends analytically on $\xi$. 

Note that in the third integral the path can be deformed
homotopically so that it does not intersect the disks
\eqref{disk-critical}. In this  case,  the distance between a
point $t'$ on the integration path and a point $'t_0-u_j('t_0)\one$
($1\leq j\leq N$) is
more than $\rho$. The singularities of the period
$I^{(0)}(t',\xi)$ are precisely at the points
\ben
t'-\xi\one = {'t}_0 - u_j('t_0)\one\,, \qquad 1\leq j\leq N\,,
\een 
which means that as long as $|\xi|<\rho$ the period will be
analytic in $\xi.$ In other words, the third integral is analytic for
$|\xi|<\rho$.

It remains only to analyze the last integral. Let us make the
substitution
\ben
t' = t_0 + x\one\quad \Rightarrow \quad u_i(t') = x\,;
\een 
then the integral becomes
\ben
\int_{\rho_i/2}^{-\ep''} (I^{(0)}_\al(t_0,\xi-x),
I^{(0)}_\beta(t_0,-x)) dx\,.
\een
Since the integration path and the corresponding $\xi$-strip (for $|\xi|<\rho$)  are
entirely in the disk of convergence, we can compute the above
integral via its Laurent series expansion
\ben
I^{(0)}_\beta(t_0,\xi-x) = \frac{\pm 2}{\sqrt{2(\xi -
   x)}}\Big( e_i + \sum_{m=1}^\infty a_m(t_0) (2(\xi -
x))^m \Big)\,, 
\een
where (see \prref{vanishing_a1})
\beq\label{am}
a_m(t_0) = \frac{\pm 2 \Psi_{t_0}R_m(t_0)e_i}{(2m-1)!!}\,,\qquad m\geq 0\,.
\eeq
Let us decompose $\al=\al'+(\al|\beta)\beta/2$,
where $\al'$ is invariant with respect to the local monodromy. Then
 $I^{(0)}_{\al'}(t',\xi)$ is analytic
in $\xi$, which implies that its contribution to the integral is
analytic, i.e., we may  replace $\al$ by $(\al|\beta)\beta/2$. 

Multiplying out the Laurent series 
and integrating term by term we get 
\ben
(\al|\beta)\, 
\int_{\rho_i/2}^{-\ep''}
\frac{dx}{\sqrt{(\xi-x)(-x)}} + O(\xi) \log (\sqrt{\xi+\ep''}+\sqrt{\ep''})\,,
\een
up to terms that  depend analytically on $\xi$ or have order
$O(\sqrt{\ep''}).$ 
The second term must vanish, because if we let $\ep''$ go twice around
$0$, then according to Lemma \ref{vanishing-period} our integral
should change by a constant, while the second term changes by a
function proportional to $\xi$. Of course, we can check the
vanishing directly (by using \eqref{am}), but then we would have to repeat the proof of
Lemma \ref{vanishing-period}. The first integral is straightforward to compute, namely it is
\ben
\left. -2(\al|\beta)\log (\sqrt{\xi-x} +\sqrt{-x})\right|_{\rho_i/2}^{\ep''}=-2(\al|\beta)\log(\sqrt{\xi+\ep''}+\sqrt{\ep''})+\cdots,
\een
where the dots indicate a function analytic in the disk
$|\xi|<\rho$. Passing to the limit we see that our integral
up to analytic terms is $-(\al|\beta)\log \xi$. This singularity cancels, up
to an integer multiple of $2\pi \sqrt{-1}$, with
the singularity \eqref{sing-0} of the first two integrals
\eqref{4-integrals}. Finally, note that if $\ep'$ or $\ep''$ makes a small
loop around $0$, then the integral gains an integer multiple of $\pm 2\pi\sqrt{-1}$. This
proves that the limit $(\ep',\ep'')\to 0$ exists and it depends
analytically on $\xi$ as claimed.
\qed

\begin{lemma}
The function $f_{\al,\beta}(t_0,\xi)$ is locally constant with respect to
$t_0.$
\end{lemma}
\proof
Let us take a small neighborhood $U$ of $t_0$ in $B$; then $u_j=u_j(t)$,
$1\leq j\leq N$ are local coordinates in $U$ and the local equation of the
discriminant in these coordinates is given by $u_i=0.$ In other words,
the critical values $(u_1,\dots,\hat{u_i},\dots,u_N)$ are local
coordinates on the discriminant near $t_0$. In the integral
representation \eqref{4-integrals} only the last two integrals depend on
$t_0$, so let us look at them more carefully. We have an improper integral on the complex plane with $N-1$
punctures:
\ben
\C\backslash\{u_1(t_0),\dots,\widehat{u_i(t_0)},\dots,
u_N(t_0)\},\quad \C\equiv \{t_0\}\times \C.
\een
 For homotopy reasons, we may think that varying $t_0$ along the
discriminant is equivalent to still integrating along the same path but 
changing the positions of the punctures $u_j(t_0)$, $j\neq
i$. However,  thanks to Lemma \ref{path-independence} the integral
does not depend on the position of the punctures.  
\qed

\subsection{Proof of Theorem \ref{t-ope}}
In order to prove that $f_{\al,\beta}$ is a constant, let us see how
the integral changes under rescaling $\xi\mapsto c\xi$, where
$c$ is a positive constant sufficiently close to $1$. 
Recall that the flat coordinates were assigned degrees $1-d_i={\rm
  deg}(\tau_i)$, where $d_i$ is the weighted-homogeneous degree of the
corresponding polynomial $\d_if$ (see Sect.\ \ref{s2}). It
follows that the structure constants $C_{ij}^k(\tau)$ of the Frobenius
multiplication, defined by
\ben
\d_i\bullet_\tau \d_j =\sum_{k=1}^N C_{ij}^k(\tau)\d_k
\een    
are homogeneous of degree $d_i+d_j-d_k$. 

Let us look at the integrand of $f_{\al,\beta}(\xi)$,
\ben
I^{(0)}_\beta(t,\xi)\bullet_t I^{(0)}_\beta(t,0) =\sum_{i,j,k=1}^N
(I^{(0)}_\beta(t,\xi),\d_k) C_{ij}^k(\tau) (I^{(0)}_\beta(t,0),d\tau_j) d\tau_i\,.
\een
Since the periods satisfy the homogeneity equation (cf.\ \eqref{periods:de1}, \eqref{periods:de2}):
\ben
(\la\d_\la +E) I^{(0)}_\beta(t,\la) = (\theta -1/2)I^{(0)}_\beta(t,\la)
\een
and 
\ben
\theta(\d_k) = (d/2-d_k)\d_k\,,\quad \theta(d\tau_j)  =(d_j-d/2)d\tau_j\,,
\quad \leftexp{T}{\theta}=-\theta\,,
\een
we see that $(I^{(0)}_\beta(t,\xi),\d_k)$ and
$(I^{(0)}_\beta(t,0),d\tau_j)$ are homogeneous of degrees 
$d_k-(d+1)/2$ and $-d_j+(d-1)/2$, respectively. It follows that the degree of the
integrand is 
\ben
d_k-(d+1)/2 + d_i+d_j-d_k -d_j+(d-1)/2+(1-d_i) = 0\,.
\een
This means that our function $f_{\al,\beta}(\xi)$ is homogeneous of
degree $0$, i.e., $f_{\al,\beta}(c\xi)=f_{\al,\beta}(\xi)$ for all
$c$ sufficiently close to $1$. Therefore, only the constant
term in the Taylor series expansion of $f_{\al,\beta}(\xi)$ at
$\xi=0$  could be non-zero.
This completes the proof of \thref{t-ope}.

\section{$\W$-Constraints }\label{s5}
In this section we prove our main result, Theorem \ref{t1}. 
The proof consists of several simple steps.

\subsection{Reduction to analyticity at the critical values}

Recall that for each $t\in B$, the operator series
$Y_t^\infty(a,\la)$ defined in Sect. \ref{sperep} provide a
$\si$-twisted  representation of the vertex algebra  $\F$ 
on the twisted Fock space $\F_\hbar$. Let us complete $\F_\hbar$ with respect to the formal
topology near $\q=-\one z$, i.e., define
\ben
\overline{\F}_\hbar
=\C_\hbar[[q_0,q_1+\one,q_2,\dots]] \,.
\een
Then the action of $Y_t^\infty(a,\la)$ on elements of $\overline{\F}_\hbar$ still makes
sense, since the operator series are given by the Wick formula
\eqref{Wick} and the periods $I^{(n)}_\al(t,\la)\to 0$ in the formal
$\la^{-1}$-adic topology as $n\to\infty.$

By definition, the total descendant potential is an element of yet
another completion of the twisted Fock space, namely 
\ben
 \D_{X_N} \in \C_\hbar[[q_0-\tau,q_1+\one,q_2,\dots]] = \widehat{\S}_t^{-1} \, \overline{\F}_\hbar
\een 
(see \eqref{dxna}, \eqref{ancestor} and \leref{lstiso}).
The latter is equipped with the structure of a $\si$-twisted $\F$-module via the operator series
\ben
Y(a,\la):=Y_0^\infty(a,\la)\,,\qquad a\in \F \,. 
\een 
If $a\in \W_{X_N}$, then due to the $\si$-invariance, $Y(a,\la)$ has only integral powers of $\la$.  
\thref{t1} is equivalent to the statement that
$Y(a,\la)\D_{X_N}$ has no negative powers of $\la$ for all $a\in\W_{X_N}$.

Let us assume that $a\in \W_{X_N}$, $t\in B$ is generic and $|\la|$ is sufficiently large, so
that the Laurent series expansions are convergent. By \leref{S}, we have
\ben
Y(a,\la)\, \widehat{\S}^{-1}_t = \widehat{\S}^{-1}_t \,
Y^\infty_t(a,\la) \,,
\een
which together with \eqref{dxna} gives
\ben
Y(a,\la)\D_{X_N} = e^{F^{(1)}(t)} \widehat{\mathcal{S}}_t^{-1} Y_t^\infty(a,\la) \A_t \,,
\een
where $\A_t$ is the ancestor potential \eqref{ancestor}.
Thus, $Y(a,\la)\D_{X_N}$ has no negative powers of $\la$ 
if and only $Y_t^\infty(a,\la) \A_t$ has the same property.

Since $\A_t$ is tame, the coefficients of $X_t(a,\la)\A_t$ 
 are polynomial expressions in the coefficients of
$X_t(a,\la)$, i.e., $Y_t^\infty(a,\la)\A_t$ is the Laurent series expansion
of $X_t(a,\la)\A_t$ at $\la=\infty.$ We want to show that all coefficients of the series
$X_t(a,\la)\A_t$ are polynomials in
$\la$. Since $X_t(a,\la)\A_t$ is monodromy
invariant and has singularities only at the critical values, 
this is equivalent to the condition that $X_t(a,\la)\A_t$ does not have poles at the critical
values $\la=u_i(t)$ for $1\leq i\leq N.$
 
\subsection{Reduction to the case of Virasoro constraints of an $A_1$-singularity}

For $a\in \W_{X_N}$, we
analytically continue $X_t(a,\la)$ to a neighborhood of
$\la=u_i(t)$. The vanishing cycle over the point $(t,u_i(t)) \in\Si$ will be denoted by
$\beta$. 

Recall that $\F_\beta\subset \F$ is the subalgebra generated by $\be$, and $\F_\be^\perp$
is the subalgebra generated by all $\al\in\lieh$ such that $(\al|\be)=0$ (see \eqref{fbeta}, \eqref{fbetaperp}).
\comment{
Put $\F_\beta\subset \F$ for the vector subspace spanned by all
polynomial expressions in $T^n\beta, n\in \Z$, where $T=-\d_t$ is the translation
operator in $\F$. Similarly let $\F_\beta^\perp\subset\F$ be the vector
subspace spanned by all polynomial expressions in $T^n\al$, where
$n\in \Z$ and  $(\al|\beta)=0.$ If we choose a basis
$\al=(\al_1,\dots,\al_{N-1})$ of the orthogonal complement $(\C\beta)^\perp$
then we have the following isomorphisms:
\ben
\F_\beta^\perp \cong \C[\al,T\al,\dots ],\quad
\F_\beta\cong \C[\beta,T\beta,\dots]. 
\een
Note that the $(-1)$-product in $\F$
induces the following isomorphism:
\ben
\F_\beta^\perp\otimes\F_\beta\cong \F,\quad v\otimes w\mapsto v_{(-1)}w.
\een
Let us decompose the vector $a$ according to the above isomorphism
}
Then, according to \eqref{fbetaiso}, for every $a\in\F$ we can write
\ben
a=\sum_k a^k_{(-1)} b^k \,, \qquad 
a^k \in \F_\be^\perp \,, \; b^k \in \F_\be \,.
\een

\begin{lemma}\label{lfbeperp}
Assume that\/ $a\in\F$ is written as above with linearly independent\/ $a^k$.
Then\/ ${e^\beta}_{(0)}a=0$ if and only if all\/ $b^k$ lie in the Virasoro vertex algebra
generated by\/ $\om_\be :=\be_{(-1)}\be/4$.
\end{lemma}
\begin{proof}
Since ${e^\beta}_{(0)}$ is a derivation of the $(-1)$-st product and ${e^\beta}_{(0)}a^k = 0$, we see that
${e^\beta}_{(0)}a=0$ if and only if ${e^\beta}_{(0)}b^k=0$ for all $k$. This means that $b^k$
lie in the $\W_{A_1}$-algebra corresponding to the lattice $\Z\be$, which is just the Virasoro vertex algebra
(cf.\ \reref{rwn}).
\end{proof}


By \thref{txtla2} (a) and (c), the operators $X_t(a^k,\la)$ are regular at $\la=u_i$ and 
\ben
Y^{u_i}_t(a,\la) \A_t = \sum_k Y^{u_i}_t(a^k,\la) \, Y^{u_i}_t(b^k,\la) \A_t \,,
\een
where $Y^{u_i}_t(a,\la)$ denotes the Laurent series expansion at
$\la=u_i$ of $X_t(a,\la)$.
If we prove that $Y^{u_i}_t(b^k,\la)\A_t$ does not have a pole at
$\la=u_i$, then the above expression does not have a pole as well. 

Due to \leref{R}, 
\ben
Y^{u_i}_t(b,\la) \, \widehat{\Psi}_t\widehat{R}_t = \widehat{\Psi}_t\widehat{R}_t \,
Y^{A_1}_{u_i}(b,\la) \,, \qquad b\in \F_\be\,,
\een
where the operator series $Y^{A_1}_{u_i}(b,\la)$ provide an 
$r_\be$-twisted representation of the vertex algebra $\F_\be$ on $\F_\hbar^{\rm tame}$
(see the discussion before \leref{R}).
Then \eqref{ancestor} implies that for $b\in \F_\beta$,
\ben
Y^{u_i}_t(b,\la) \A_t(\q) = \widehat{\mathcal{R}}_t\, Y^{A_1}_{u_i}(b,\la) 
\prod_{j=1}^N \D_{\rm pt}(\hbar\Delta_j;\leftexp{j}{\q}) \,,
\een
where $Y^{A_1}_{u_i}(b,\la)$ is acting on the $i$-th factor in the product.

When $b$ is in the Virasoro vertex algebra generated by $\om_\be$,
the operators $Y^{A_1}_{u_i}(b,\la)$ give an untwisted representation.
By \coref{clocpr}, the regularity of $Y^{A_1}_{u_i}(b,\la)$ at $\la=u_i$
follows from the regularity of the generating field $Y^{A_1}_{u_i}(\om_\be,\la)$.
Therefore, we only need to verify that the Virasoro
constraints for an $A_1$-singularity coincide with the usual Virasoro
constraints for the Witten--Kontsevich tau-function $\D_{\rm pt}$.

\subsection{Virasoro constraints for an $A_1$-singularity}

Let us assume now that $F(t,x)=x^2/2+t$ is the miniversal deformation
of an $A_1$-singularity. Then the period has the form
\ben
I_\beta^{(0)}(t,\la) = \frac{2}{\sqrt{2(\la-t)}} \,,
\een
where the vanishing cycle over the point $(t,\la)\in B\times \C$ is
the $0$-dimensional cycle 
\ben
\beta = [x_+(t,\la)]-[x_-(t,\la)]\,,\qquad x_\pm(t,\la) = \pm \sqrt{2(\la-t)}\,.
\een
{}From here we find for $k\geq0$
\begin{align*}
(-1)^{k+1} I_\beta^{(k+1)}(t,\la) &= (-\d_\la)^{k+1} I_\beta^{(0)}(t,\la) =
2^{-k-\frac12} (2k+1)!! \, (\la-t)^{-k-\frac32} \,,
\\
 I_\beta^{(-k)}(t,\la) &= (\d_\la)^{-k} I_\beta^{(0)}(t,\la) =
\frac{ 2^{k+\frac12} }{ (2k-1)!! } \, (\la-t)^{k-\frac12} \,,
\end{align*}
where $(-1)!! :=1$.
After quantization (see \eqref{quant1}, \eqref{quant2}), we obtain the differential operator
\ben
X_t(\beta,\la) = \d_\la \widehat\f_\beta(t,\la) =
\sum_{n:{\rm odd}} J_n (\la-t)^{-\frac{n}{2}-1} \,,
\een
where for $k=0,1,2,\dots$
\ben
J_{2k+1} =   2^{-k-\frac12} (2k+1)!! \,\hbar^{1/2} \frac{\d}{\d q_k} \,,\qquad
J_{-2k-1}  =  \frac{2^{k+\frac12} }{(2k-1)!!}  \, \hbar^{-1/2} q_k\,.
\een
%
%
The formula for the propagator \eqref{propser}, \eqref{prop-seru} assumes the form
\begin{align*}
P_{\beta,\beta}(t,\la;\mu-\la) &=
\sum_{k=0}^\infty (2k+1) (\mu-t)^{-k-\frac32}  (\la-t)^{k-\frac12}
\\
&= -2\iota_{\mu-t,\la-t} \d_\mu \bigl( (\mu-\la)^{-1} (\mu-t)^{1/2} (\la-t)^{-1/2} \bigr) \,. 
\end{align*}
On the other hand, we have
\ben
X_t(\om_\be,\la) = \frac14 X_t(\be_{(-1)}\be,\la) =
\frac14 {:} X_t(\beta,\la)X_t(\beta,\la) {:} + \frac14 P^0_{\beta,\beta}(t,\la)\,.
\een
After a short computation, we find
$P^0_{\beta,\beta}(t,\la) = (\la-t)^{-2}/4$, which implies
\ben
X_t(\om_\be,\la) 
= \sum_{m\in \Z} L_m (\la-t)^{-m-2} \,,
\een
where the Virasoro operators are
\ben
L_{m} = \frac{1}{16} \delta_{m,0} + \frac14 \sum_{k\in \Z} {:} J_{2(k+m)+1} J_{-2k-1} {:} \,. 
\een

For example, the first few operators are as follows:
\begin{align*}
L_{-1} & =  \frac1{2\hbar} {q_0}^2 + \sum_{k=0}^\infty q_{k+1} \frac{\d}{\d q_k} \,,\\
L_0 & =  \frac1{16}+ \frac12 \sum_{k=0}^\infty (2k+1) q_k \frac{\d}{\d q_k} \,,\\
L_1 & =  \frac{\hbar}8 \frac{\d^2}{{\d q_0}^2} + \frac14 \sum_{k=0}^\infty
(2k+3)(2k+1)q_k\frac{\d}{\d q_{k+1}} \,,\\
L_2 & =  \frac{3\hbar}8 \frac{\d^2}{\d q_0\d q_1} + \frac18 \sum_{k=0}^\infty 
(2k+5)(2k+3)(2k+1)q_k\frac{\d}{\d q_{k+2}} \,.
\end{align*}
After setting $\hbar=1$, these become precisely the Virasoro operators that characterize the
Witten--Kontsevich tau-function (see \cite{W1b}). 

Finally, let us point out that the total descendant potential is
obtained from a product of formal power series that are obtained
from the Witten--Kontsevich tau-function by rescaling:
\ben
\D_{\rm pt}(\hbar,\q)\mapsto \D_{\rm pt}(\hbar \Delta_i,\q\sqrt{\Delta_i}) \,.
\een
Since the above Virasoro operators are invariant under such a rescaling, the rescaled potentials
still satisfy the same Virasoro constraints. 

This completes the proof of Theorem \ref{t1}.

\section*{Acknowledgements}
The authors benefited from attending the AIM workshop
``Integrable systems in Gromov--Witten and symplectic field theory"
(January 30 -- February 3, 2012)
and, in particular, from discussions with B.\ Dubrovin and S.\ Shadrin.
The second author benefited from conversations with  E.\ Frenkel, A.\
Givental, and K.\ Saito. The first author is supported in part by NSF
and NSA grants  
and the second author is supported in part by a JSPS Grant-In-Aid.

\bibliographystyle{amsalpha}

\end{document}